%% file: online-CO-v1.tex
\documentclass[letterpaper, 10 pt, conference]{ieeeconf}  % Comment this line out
\IEEEoverridecommandlockouts                              % This command is only
\input{styledsymbols}

\usepackage{epsfig} % for postscript graphics files
\usepackage{epstopdf}

\usepackage{fancyhdr}
\usepackage{indentfirst}
\usepackage{amsmath}
\usepackage{amssymb}
\usepackage{algorithm}
\usepackage{algorithmic}

\usepackage{amsthm}

\usepackage{graphicx}
\usepackage{color}

\definecolor{gray}{RGB}{128,128,128}

\usepackage[short]{optidef}
\usepackage{latexsym}
\usepackage{amsfonts}
\usepackage{amssymb,amsmath,amsfonts}
\usepackage{amsmath,mathrsfs,bm,url,times}
\usepackage{cite}
\usepackage{caption}
\usepackage{subcaption}
\newtheorem{assumption}{Assumption}

\newtheorem{theorem}{Theorem}
\newtheorem{lemma}{Lemma}

\newtheorem{definition}{Definition}

\newtheorem{remark}{Remark}
\newtheorem{corollary}{Corollary}

\DeclareMathOperator{\Dom}{Dom}

\DeclareMathOperator*{\argmin}{arg\,min}

\DeclareMathOperator{\Reg}{Reg}
\DeclareMathOperator{\col}{col}
\DeclareMathOperator{\inout}{in}
\DeclareMathOperator{\outin}{out}
\usepackage{array}
\newcolumntype{M}[1]{>{\centering\arraybackslash}m{#1}}
\newcolumntype{N}{@{}m{0pt}@{}}

\allowdisplaybreaks
\title{\LARGE \bf
Distributed Online Convex Optimization\\ with Time-Varying Coupled Inequality Constraints
}

\author{Xinlei Yi, Xiuxian Li, Lihua Xie, and Karl H. Johansson% <-this % stops a space
\thanks{This work was supported by the
Knut and Alice Wallenberg Foundation, the  Swedish Foundation for Strategic Research, and the Swedish Research Council.}% <-this % stops a space
\thanks{X. Yi and K. H. Johansson are with the Division of Decision and Control Systems, School of Electrical Engineering and Computer Science, KTH Royal Institute of Technology, 100 44, Stockholm, Sweden.
        {\tt\small \{xinleiy, kallej\}@kth.se}.}%
\thanks{X. Li and L. Xie are with School of Electrical and Electronic Engineering,
Nanyang Technological University, 50 Nanyang Avenue, Singapore 639798. {\tt\small \{xiuxianli, elhxie\}@ntu.edu.sg}.}
}

\begin{document}

\maketitle
\thispagestyle{empty}
\pagestyle{empty}

\begin{abstract}\label{online_op:Abstract}This paper considers distributed online optimization with time-varying coupled inequality constraints. The global objective function is composed of  local convex cost and regularization functions and the coupled constraint function is the sum of local convex functions. A distributed online primal-dual dynamic mirror descent algorithm is proposed to solve this problem, where the local cost, regularization, and constraint functions are held privately and revealed only after each time slot. Without assuming Slater's condition,  we first derive regret and constraint violation bounds for the algorithm and show how they depend on the stepsize sequences, the accumulated dynamic variation of the comparator sequence, the number of agents, and the network connectivity. As a result,  under some natural decreasing stepsize sequences, we prove that the algorithm achieves sublinear dynamic regret and constraint violation if the accumulated dynamic variation of the optimal sequence also grows sublinearly. We also prove that the algorithm achieves sublinear static regret and constraint violation under mild conditions. Assuming Slater's condition, we show that the algorithm achieves smaller bounds on the constraint violation. In addition, smaller bounds on the static regret are achieved when the objective function is strongly convex. Finally, numerical simulations are provided to illustrate the effectiveness of the theoretical results.

\end{abstract}
%%%%%%%%%%%%%%%%%%%%%%%%%%%%%%%%%%%%%%%%%%%%%%%%%%%%%%%%%%%%%%%%%%%%%%%%%%%%%%%%
\section{Introduction}\label{online_opsec:intro}
Consider a network of $n$ agents indexed by $i=1,\dots,n$. For each $i$, let the local decision set $X_i\subseteq\mathbb{R}^{p_i}$ be a closed convex set with $p_i$ being a positive integer. Let $\{f_{i,t}:X_i\rightarrow \mathbb{R}\}$ and $\{g_{i,t}:X_i\rightarrow \mathbb{R}^m\}$ be arbitrary sequences of local convex cost and constraint functions over time $t=1,2,\dots$, respectively, where $m$ is a positive integer. At each $t$, the network's objective is to solve the convex optimization problem $\min_{x_t\in X}\sum_{i=1}^nf_{i,t}(x_{i,t})$ with coupled constraint $\sum_{i=1}^ng_{i,t}(x_{i,t})\le{\bf0}_{m}$, where the global decision variable is $x_t=\col(x_{1,t},\dots,x_{n,t})\in X=X_1\times\cdots\times X_n\subseteq\mathbb{R}^{p}$ with $p=\sum_{i=1}^np_i$. We are interested in distributed algorithms to solve this problem, where computations are done by each agent. It is common to influence the structure of the solution using regularization. In this case, each agent $i$ introduces a regularization function $r_{i,t}:X_i\rightarrow \mathbb{R}$. Examples of regularization include $\ell_1$-regularization $r_{i,t}(x_i)=\lambda_i\|x_i\|_1$ and $\ell_2$-regularization $r_{i,t}(x_i)=\frac{\lambda_i}{2}\|x_i\|$ with $\lambda_i>0$. The global objective function now becomes $f_t(x_t)=\sum_{i=1}^n(f_{i,t}(x_{i,t})+r_{i,t}(x_{i,t}))$. Denote $g_t(x_t)=\sum_{i=1}^ng_{i,t}(x_{i,t})$.  To summarize, we are interested in solving the constrained optimization problem
\begin{mini}
{x_t\in X}{f_t(x_t)}{\label{online:problem1}}{}
\addConstraint{g_t(x_t)\le}{{\bf 0}_m,\quad}{t=1,\dots}
\end{mini}
using distributed algorithms.
In order to guarantee that problem (\ref{online:problem1}) is feasible, we assume that for any $T\in\mathbb{N}_+$, the set of all feasible sequences $
\mathcal{X}_{T}=\{(x_1,\dots,x_T):~x_t\in X,~g_{t}(x_t)\le{\bf0}_{m},~
t=1,\dots,T\}
$ is non-empty. With this standing assumption, an optimal sequence to (\ref{online:problem1}) always exists.

We consider  online algorithms. For a distributed online algorithm, at time $t$, each agent $i$ selects a decision $x_{i,t}\in X_i$. After the selection, the agent receives its cost function $f_{i,t}$ and regularization $r_{i,t}$ together with its constraint function $g_{i,t}$. At the same moment, the agents exchange data with their neighbors over a time-varying directed graph.
%We are interested in the setting that the functions $f_{t,i}$, $r_{t,i}$, and $g_{t,i}$ are held privately by agent $i$ and only revealed at the end of each time slot as well as designing distributed online algorithms to solve (\ref{online:problem1}) under this setting. Such an algorithm is a recursive algorithm in which at each time $t$ each agent $i$ determines its decision $x_{t,i}$ based on the information at the previous time $t-1$ of its local objective, regularization, and constraint functions as well as data received from its neighbors. The communication between agents is modeled by a time-varying directed graph.
The performance of an algorithm depends on both the amount of data exchanged between the agents and how they process the data.
For online algorithms, regret and constraint violation are often used as performance metrics. The regret is the accumulation over time of the loss difference between the decision determined by the algorithm and a comparator sequence. Specifically, the efficacy of a decision sequence $\bsx_T=(x_{1},\dots,x_{T})$ relative to a comparator sequence $\bsy_T=(y_1,\dots,y_T)\in\mathcal{X}_{T}$ with $y_t=\col(y_{1,t},\dots,y_{n,t})$ is characterized by the regret
\begin{align}\label{online_op:reg}
\Reg(\bsx_T,\bsy_T)=&\sum_{t=1}^Tf_{t}(x_{t})-\sum_{t=1}^Tf_{t}(y_{t}).
\end{align}
There are two special comparators. One is $\bsy_T=
\bsx^*_T=\argmin_{\bsx_T\in\calX_{T}}\sum_{t=1}^Tf_t(x_t)$, i.e., an optimal sequence to (\ref{online:problem1}).
In this case $\Reg(\bsx_T,\bsx^*_T)$ is called the dynamic regret. Another special comparator is the static optimal sequence $\bsy_T=\check{\bsx}^*_T=\argmin_{\bsx_T\in\check{\calX}_{T}}\sum_{t=1}^Tf_t(x_t)
$, where $\check{\calX}_{T}=\{(x,\dots,x):~x\in X,~g_{t}(x)\le{\bf0}_{m},~t=1,\dots,T\}\subseteq\calX_{T}$ is the set of feasible static sequences. In order to guarantee the existence of $\check{\bsx}^*_T$, we assume that $\check{\calX}_{T}$ is non-empty.
In this case $\Reg(\bsx_T,\check{\bsx}^*_T)$ is called the static regret. It is straightforward to see that $\Reg(\bsx_T,\bsy_T)\le\Reg(\bsx_T,\bsx^*_T),~\forall \bsy_T\in\calX_{T}$, and that $\Reg(\bsx_T,\check{\bsx}^*_T)\le\Reg(\bsx_T,\bsx^*_T)$.
For a decision sequence $\bsx_T$, the normally used constraint violation measure is $\|[\sum_{t=1}^Tg_{t}(x_{t})]_+\|$, i.e., the accumulation of constraint violations. This definition implicitly allows constraint violations at some times to be compensated by strictly feasible decisions at other times. This is appropriate for constraints that have a cumulative nature such as energy budgets enforced through average power constraints.

This paper develops a distributed online algorithm to solve (\ref{online:problem1}) with guaranteed performance measured by the regret and  constraint violation. We are satisfied with low regret and  constraint violation, by which we mean that both $\Reg(\bsx_T,\bsy_T)$ and $\|[\sum_{t=1}^Tg_{t}(x_{t})]_+\|$ grow sublinearly with $T$, i.e., there exist $\kappa_1,~\kappa_2\in(0,1)$ such that $\Reg(\bsx_T,\bsy_T)=\mathcal{O}(T^{\kappa_1})$ and $\|[\sum_{t=1}^Tg_{t}(x_{t})]_+\|=\mathcal{O}(T^{\kappa_2})$. This implies that the upper bound of the time averaged difference between the accumulated cost of the decision sequence and the accumulated cost of any comparator sequences tends to zero as $T$ goes to infinity. The same thing holds for the upper bound of the time averaged constraint violation. The novel algorithm we design explores the stepsize sequences in a way that allows the trade-off between how fast these two bounds tend to zero.

\subsection{Motivating Example}\label{online_op:example}
As a motivating example, consider a multi-target tracking problem in which $n$ agents follow $n$ targets. Let $z_i(s),~\tilde{z}_i(s)$ denote the positions of agent $i$ and target $i$ at time $s$, respectively. To model agent and target paths, we introduce a parameterization:
\begin{align*}
&z_i(s)=\sum_{k=1}^{p_i}x_{i,t}[k]c_{k,t}(s),\\
&\tilde{z}_i(s)=\sum_{k=1}^{p_i}y_{i,t}[k]c_{k,t}(s),~s\in[t,t+1),
\end{align*}
where $c_{k,t}(s)$ are vector functions that parameterize the space of possible trajectories over time $[t,t+1)$ and satisfy
\begin{align*}
\int_{t}^{t+1}\langle c_{k,t}(s),c_{l,t}(s)\rangle ds=\begin{cases}
1,~&\text{if}~k=l\\
0,~&\text{else.}
\end{cases}
\end{align*}
The action spaces of agent $i$ and target $i$ are given by  $x_{i,t}=[x_{i,t}[1],\dots,x_{i,t}[p_i]]^\top\in X_i\subseteq\mathbb{R}^{p_i}$ and $y_{i,t}=[y_{i,t}[1],\dots,y_{i,t}[p_i]]^\top\in\mathbb{R}^{p_i}$, respectively. At time $t$, agent $i$ repositions itself by selecting an action $x_{i,t}$ such that it could stay as close as possible to target $i$ during time $[t,t+1)$ and at the same time it wants the selection cost $\langle \pi_{i,t},x_{i,t}\rangle$ to be as small as possible, where $\pi_{i,t}\in\mathbb{R}^{p_i}_+$ is the price vector. This goal can be captured by defining a local cost function
\begin{align*}
f_{i,t}(x_{i,t})&=\zeta_{i,1}\langle \pi_{i,t},x_{i,t}\rangle+\zeta_{i,2}\int_{t}^{t+1}\|z_i(s)-\tilde{z}_i(s)\|^2ds\\
&=\zeta_{i,1}\langle \pi_{i,t},x_{i,t}\rangle+\zeta_{i,2}\|x_{i,t}-y_{i,t}\|^2,
\end{align*}
where $\zeta_{i,1}$ and $\zeta_{i,2}$ are nonnegative constants to trade-off the two subgoals.  Here, target $i$'s action $y_{i,t}$ and the price vector $\pi_{i,t}$ are observed only after the selection.
Agents need to cooperatively take into account energy and communication constraints. For simplicity, we introduce linear local constraint functions $g_{i,t}(x_{i,t})=D_{i,t}x_{i,t}-d_{i,t}$, where $D_{i,t}\in\mathbb{R}^{m\times p_i}$ and $d_{i,t}\in\mathbb{R}^{p_i}$ are time-varying and unknown at time $t$. These coupling
constraints determine the limits on the available resources to be shared among the agents.
Section~\ref{online_opsec:simulation} shows how this multi-target tracking problem can be solved by the algorithm proposed in this paper.

\subsection{Literature Review}
The online optimization problem (\ref{online:problem1}) is related to two bodies of literature: centralized online convex optimization with time-varying inequality constraints ($n=1$) and distributed online convex optimization with time-varying coupled inequality constraints ($n\ge2$). Depending on the characteristics of the constraint, there are two important special cases: optimization with static constraints ($g_{i,t}\equiv0$ for all $t$ and $i$) and time-invariant constraints ($g_{i,t}\equiv g_i$ for all $t$ and $i$). Below, we provide an overview of the related works.
%Problem (\ref{online:problem1}) is related to different bodies of literature from different perspectives.  From the perspective of $n$, the number of agents, problem (\ref{online:problem1}) is centralized (distributed) online convex optimization if $n=1$ ($n\ge2$). From the perspective of $g_{t,i}$, the constraint function, there are three scenarios. The first scenario is that problem (\ref{online:problem1}) is online convex optimization with static set constraints if $g_{t,i}\equiv0$ for all $t$ and $i$. The second scenario is that problem (\ref{online:problem1}) is online convex optimization with time-invariant inequality constraints if $g_{t,i}\equiv g_i$ for all $t$ and $i$, where $g_i$ is a function. The final scenario is that problem (\ref{online:problem1}) is online convex optimization with time-varying inequality constraints. Below, we provide an overview of the related works.

Centralized online convex optimization with static set constraints was first studied by Zinkevich \cite{zinkevich2003online}. Specifically, he developed a projection-based online gradient descent algorithm and achieved $\mathcal{O}(\sqrt{T})$ static regret bound for an arbitrary sequence of convex objective functions with bounded subgradients. It was later shown that this is a tight bound up to constant factors \cite{hazan2007logarithmic}. The regret bound can be reduced under more stringent strong convexity conditions on the objective functions \cite{hazan2007logarithmic,shalev2012online,hazan2016introduction,mokhtari2016online} or by allowing to query the gradient of the objective function multiple times \cite{zhang2017improved}. When the static constrained sets are characterized by inequalities, the conventional projection-based online algorithms are difficult to implement and may be inefficient in practice due to high computational complexity of the projection operation. To overcome these difficulties, some researchers proposed primal-dual algorithms for centralized online convex optimization with time-invariant inequality constraints, e.g., \cite{mahdavi2012trading,jenatton2016adaptive,yu2016low,NIPS2018_7852}.
The authors of \cite{sun2017safety} showed that the algorithms proposed in \cite{mahdavi2012trading,jenatton2016adaptive} are general enough to handle time-varying inequality constraints. The authors of \cite{chen2017online} used the modified saddle-point method to handle time-varying constraints. The papers \cite{yu2017online,neely2017online} used a virtual queue, which essentially is a modified Lagrange multiplier, to handle stochastic and time-varying constraints and the authors of \cite{chen2019bandit} extended the algorithm proposed in \cite{neely2017online} with bandit feedback. The authors of \cite{paternain2017online} studied online convex optimization with time-varying constraints in the continuous-time setting and showed that the static regret in continuous-time can be bounded by a constant independent of the time horizon, as opposed to the sublinear static regret observed in the discrete-time setting.

Distributed online convex optimization has been extensively studied, so here we only list some of the most relevant work. Firstly, the authors of \cite{tsianos2012distributed,yan2013distributed,koppel2015saddle,hosseini2016online,shahrampour2018distributed,yuan2019distributed}
proposed distributed online algorithms to solve convex optimization problems with static set constraints and achieved sublinear regret. For instance, the authors of \cite{shahrampour2018distributed} proposed a decentralized variant of the dynamic mirror descent algorithm proposed in \cite{hall2015online}. Mirror descent generalizes classical gradient descent to Bregman divergences and is suitable for solving high-dimensional convex optimization problems. The weighted majority algorithm in machine learning \cite{littlestone1994weighted} can  be viewed as a special case of mirror descent. Secondly, the paper \cite{yuan2017adaptive} extended the adaptive algorithm proposed in \cite{jenatton2016adaptive} to a distributed setting to solve an online convex optimization problem with a static inequality constraint.  Finally, the authors of \cite{lee2017sublinear,Li2018distributed} proposed distributed primal-dual algorithms to solve an online convex optimization with static coupled inequality constraints. To the best of our knowledge, no papers considered distributed online convex optimization with time-varying constraints in the discrete-time setting. In continuous-time, the authors of \cite{Paternain2019distributed} extended the online saddle point algorithm proposed in \cite{paternain2017online} to a distributed version.

\subsection{Main Contributions}
%To the best of our knowledge, this is the first work on distributed online convex optimization with time-varying coupled inequality constraints.

%This paper builds on and extends the algorithms and results in \cite{mahdavi2012trading,jenatton2016adaptive,hall2015online,shahrampour2018distributed,neely2017online,sun2017safety,chen2017online,lee2017sublinear,Li2018distributed}. The key challenge in this paper is to handle the time-varying coupled inequality constraints under the condition that only iteration index is available to design stepsizes. Compared with those papers the contributions of this paper are summarized as follows.

Compared to the literature the contributions of this paper are summarized as follows.

1) We propose a novel distributed online primal-dual dynamic mirror descent algorithm to solve the constrained optimization problem (\ref{online:problem1}). In this algorithm, each agent $i$ maintains two local sequences: the local decision sequence $\{x_{i,t}\}\subseteq X_i$ and the local dual variable sequence $\{q_{i,t}\}\subseteq\mathbb{R}^m_+$. An agent averages its local dual variable with its in-neighbors in a consensus step, and takes into account the estimated dynamics of the optimal sequences. The proposed algorithm uses different non-increasing stepsize sequences $\{\alpha_t>0\}$ and $\{\gamma_t>0\}$ for the primal and dual updates, respectively, and a non-increasing sequence $\{\beta_t>0\}$ to design penalty terms such that the dual variables are not growing too large. These sequences give some freedom in the regret and constraint violation bounds, as they allow the trade-off between how fast these two bounds tend to zero. The algorithm uses the subgradients of the local cost and constraint functions at the previous decision, but the total number of iterations or any other parameters related to the objective or constraint functions are not used.  %Moreover,  in order to influence the structure of the decision, in the primal update the regularization function is not linearized.

2) Without assuming Slater's condition, i.e., that the feasible region has an interior point, we derive regret and constraint violation bounds for the algorithm and show how they depend on the stepsize sequences, the accumulated dynamic variation of the comparator sequence, the number of agents, and the network connectivity. The same regret bound was achieved by the centralized dynamic mirror descent proposed in \cite{hall2015online} for static set constraints.
With the stepsize sequences $\alpha_t=1/t^{c}$, $\beta_t=1/t^\kappa$, $\gamma_t=1/t^{(1-\kappa)}$, where $c,~\kappa\in(0,1)$ are user-defined trade-off parameters, we prove that our algorithm simultaneously achieves sublinear dynamic regret and constraint violation if the accumulated dynamic variation of the optimal sequence grows sublinearly.  Moreover, if $c=\kappa$ we show that the algorithm achieves the same sublinear static regret and constraint violation bounds as in \cite{jenatton2016adaptive}, i.e., $\Reg(\bsx_T,\check{\bsx}_T^*)=\mathcal{O}(T^{\max\{1-\kappa,\kappa\}})$ and $\|[\sum_{t=1}^Tg_{t}(x_{t})]_+\|=\mathcal{O}(T^{1-\kappa/2})$.  Compared with  \cite{mahdavi2012trading,jenatton2016adaptive,sun2017safety,NIPS2018_7852,Li2018distributed}, which assumed the same assumption on the cost and constraint functions as this paper, the proposed algorithm has the following advantages. The parameter $\kappa$ enables the user to trade-off static regret bound for constraint violation bound, while recovering the $\mathcal{O}(\sqrt{T})$ static regret bound and $\mathcal{O}(T^{3/4})$ constraint violation bound from \cite{mahdavi2012trading,sun2017safety} as special cases. The algorithms proposed in \cite{mahdavi2012trading,jenatton2016adaptive,sun2017safety} are centralized and the constraint functions in \cite{mahdavi2012trading,jenatton2016adaptive} are time-invariant. Moreover, in \cite{mahdavi2012trading,sun2017safety} the total number of iterations and in \cite{mahdavi2012trading,jenatton2016adaptive,sun2017safety} the upper bounds of the objective and constraint functions and their subgradients need to be known in advance to design the stepsizes.
The proposed algorithm achieves smaller static regret and constraint violation bounds than \cite{Li2018distributed}, although time-invariant coupled inequality constraints were considered. The algorithm proposed in \cite{NIPS2018_7852} achieved a better constraint violation bound than ours, but their algorithm is centralized and the  constraint function is time-invariant.

3) Assuming Slater's condition and the stepsize sequences above with $c=1-\kappa$, we show that the dynamic regret bound is similar to the bound without Slater's condition, but the constraint violation bound can be reduced to $\mathcal{O}(T^{\max\{1-\kappa,\kappa\}})$. Our results are superior to \cite{chen2017online} in the sense that the accumulated variation of constraints, $V(\{g_i\}_{t=1}^T)=\sum_{t=1}^T\max_{x\in\calX}\|[g_{t+1}(x)-g_t(x)]_+\|$, appears in their bounds and more assumptions are needed. We show that our algorithm simultaneously achieves sublinear dynamic regret and constraint violation, if the accumulated variation of the optimal sequence grows sublinearly. Moreover, the static regret and constraint violation bounds grow as $\mathcal{O}(\sqrt{T})$, which is better than the results for the centralized algorithm in \cite{neely2017online}. The authors of \cite{lee2017sublinear} achieved the same bounds, but they assumed that the coupled inequality constraints are time-invariant and they explicitly assumed boundedness of the dual variable sequence. The conditions to guarantee this assumption are not so obvious since the dual variable sequence is generated by the algorithm. In this paper, we show that the dual variable sequence is indeed bounded.

4) When the local objective functions are assumed to be strongly convex, we show that, also without Slater's condition, the proposed algorithm achieves $\mathcal{O}(T^{\kappa})$ static regret bound and $\mathcal{O}(T^{1-\kappa/2})$ constraint violation bound. Moreover, we find that the constraint violation bound can be reduced to $\mathcal{O}(T^{\max\{1-\kappa,\kappa\}})$ when Slater's condition holds.

The comparison between this paper and the literature is summarized in Table~\ref{online_op:table}.
\begin{table*}[t]
\caption{Comparison of this paper to some related works on online convex optimization.}
\label{online_op:table}
\vskip 0.15in
\begin{center}
\begin{small}
%\begin{sc}
\begin{tabular}{M{1.3cm}|M{1.72cm}|M{2.65cm}|M{10.35cm}N}
\hline
References&Problem type&Constraint type&Regret and constraint violation bounds&\\[7pt]
\hline
\cite{mahdavi2012trading}&Centralized&$g(x)\le{\bf 0}_m$
& $\Reg(\bsx_T,\check{\bsx}^*_T)\le  \mathcal{O}(\sqrt{T})$, $\|[\sum_{t=1}^Tg(x_{t})]_+\|\le \mathcal{O}(T^{3/4})$ &\\[12pt]
\hline
\cite{jenatton2016adaptive}&Centralized&$g(x)\le{\bf 0}_m$
&$\Reg(\bsx_T,\check{\bsx}^*_T)\le  \mathcal{O}(T^{\max\{1-\kappa,\kappa\}})$, $\|[\sum_{t=1}^Tg(x_{t})]_+\|\le \mathcal{O}(T^{1-\kappa/2}),\kappa\in(0,1)$&\\[12pt]
\hline
\cite{NIPS2018_7852} &Centralized&$g(x)\le{\bf 0}_m$
& $\Reg(\bsx_T,\check{\bsx}^*_T)\le  \mathcal{O}(\sqrt{T})$, $\sum_{t=1}^T\|[g(x_{t})]_+\|^2\le \mathcal{O}(\sqrt{T})$ &\\[12pt]
\hline
\cite{sun2017safety} &Centralized&$g_t(x)\le{\bf 0}_m$
&$\Reg(\bsx_T,\check{\bsx}^*_T)\le  \mathcal{O}(\sqrt{T})$, $\|[\sum_{t=1}^Tg_{t}(x_{t})]_+\|\le \mathcal{O}(T^{3/4})$ &\\[12pt]
\hline
\cite{chen2017online} &Centralized&$g_t(x)\le{\bf 0}_m$  and Slater's condition
&$\Reg(\bsx_T,\bsx^*_T)\le\mathcal{O}(\max\{T^{1/3}\sum_{t=1}^{T}\|x^*_{t}-x^*_{t-1}\|, T^{1/3}V(\{g_i\}_{t=1}^T),T^{2/3}\})$,
$~~~~~~~~~~~~~~~~~~~~~~~~~~~~~~~~~~~~~~~~~~~~~~~~~~~~~~~~~~~~~~~~~~~~~~~~~$  $\|[\sum_{t=1}^Tg_{t}(x_{t})]_+\|\le \mathcal{O}(T^{2/3})$, %if Slater's condition holds and the slack constant is larger than the point-wise maximum variation of consecutive constraints
&\\[31pt]
\hline
\cite{neely2017online} &Centralized&$g_t(x)\le{\bf 0}_m$ and Slater's condition
&$\Reg(\bsx_T,\check{\bsx}^*_T)/T\le c\epsilon$ and $\|[\sum_{t=1}^Tg_{t}(x_{t})]_+\|/T\le c\epsilon$ for $T\ge1/\epsilon^2$  &\\[14pt]
\hline
\cite{lee2017sublinear} &Distributed&$g(x)=\sum_{i=1}^ng_{i}(x_i)\le{\bf 0}_m$
&$\Reg(\bsx_T,\check{\bsx}^*_T)\le  \mathcal{O}(\sqrt{T})$, $\|[\sum_{t=1}^Tg(x_{t})]_+\|\le \mathcal{O}(\sqrt{T})$ if dual variables generated by the proposed algorithm are bounded &\\[14pt]
\hline
\cite{Li2018distributed} &Distributed&$g(x)=\sum_{i=1}^ng_{i}(x_i)\le{\bf 0}_m$
&$\Reg(\bsx_T,\check{\bsx}^*_T)\le  \mathcal{O}(T^{1/2+2\kappa})$, $\|[\sum_{t=1}^Tg(x_{t})]_+\|\le \mathcal{O}(T^{1-\kappa/2}),\kappa\in(0,1/4)$&\\[14pt]
\hline
This paper &Distributed&$g_t(x)=\sum_{i=1}^ng_{i,t}(x_i)\le{\bf 0}_m$
&$\Reg(\bsx_T,\bsx^*_T)\le  \mathcal{O}(\max\{T^{\kappa}\sum_{t=1}^{T-1}\|x^*_{t+1}-x^*_{t}\|,T^{\max\{1-\kappa,\kappa\}}\})$, $~~~~~~~~~~~~~~~~~~~~~~~~~~~~~~~~~~~~~~~~~~~~~~~~~~~~~~~~~~~~~~~~~~~~~~~~~$ $\|[\sum_{t=1}^Tg_{t}(x_{t})]_+\|\le \mathcal{O}(T^{1-\kappa/2})$ (without Slater's condition),
 $~~~~~~~~~~~~~~~~~~~~~~~~~~~~~~~~~~~~~~~~~~~~~~~~~~~~~~~~~~~~~~~~~~~~~~~~~$ $\|[\sum_{t=1}^Tg_{t}(x_{t})]_+\|\le \mathcal{O}(T^{\max\{1-\kappa,\kappa\}})$ (with Slater's condition), $\kappa\in(0,1)$ &\\[50pt]
\hline
\end{tabular}
%\end{sc}
\end{small}
\end{center}
\vskip -0.1in
\end{table*}

\subsection{Outline}
The rest of this paper is organized as follows. Section~\ref{online_opsec:preliminaries} introduces the preliminaries. Section~\ref{online_opsec:algorithm} provides the distributed primal-dual dynamic mirror descent algorithm.  Section~\ref{online_opsec:main} analyses the bounds of the regret and  constraint violation for the algorithm. Section~\ref{online_opsec:simulation} gives simulation examples. Finally, Section~\ref{online_opsec:conclusion} concludes the paper. Proofs are given in the Appendix.

\noindent {\bf Notations}: All inequalities and equalities are understood componentwise. $\mathbb{R}^n$ and $\mathbb{R}^n_+$ stand for the set of $n$-dimensional vectors and nonnegative vectors, respectively. $\mathbb{N}_+$ denotes the set of positive integers. $[n]$ represents the set $\{1,\dots,n\}$ for any $n\in\mathbb{N}_+$. $\|\cdot\|$ ($\|\cdot\|_1$) denotes the Euclidean norm (1-norm) for vectors and the induced 2-norm (1-norm) for matrices. $\langle x,y\rangle$ represents the standard inner product of two vectors $x$ and $y$. $x^\top$ is the transpose of the vector or matrix $x$. $I_n$ is the $n$-dimensional identity matrix.  ${\bf 1}_n$ (${\bf 0}_n$) denotes the column one (zero)
vector of dimension $n$. $\col(z_1,\dots,z_k)$ is the concatenated column vector of vectors $z_i\in\mathbb{R}^{n_i},~i\in[k]$. $[z]_+$ represents the component-wise projection of a vector $z\in\mathbb{R}^n$ onto $\mathbb{R}^n_+$. $\lceil \cdot\rceil$ and $\lfloor\cdot\rfloor$ denote the ceiling and floor functions, respectively. $\log(\cdot)$ is the natural logarithm.  Given two scalar sequences $\{\alpha_t,~t\in\mathbb{N}_+\}$ and $\{\beta_t>0,~t\in\mathbb{N}_+\}$, $\alpha_t=\mathcal{O}(\beta_t)$ means that there exists a constant $a>0$ such that $\alpha_t\le a\beta_t$ for all $t$, while $\alpha_t=\mathbf{o}(t)$  means that there exist two constants $a>0$ and $\kappa\in(0,1)$ such that $\alpha_t\le at^\kappa$ for all $t$.%A function $r:\mathbb{R}^n\rightarrow\mathbb{R}_{+}$ is $\ell_1$- ($\ell_2$-) regularization if $r(x)=\lambda\|x\|_1$ ($r(x)=\frac{\lambda}{2}\|x\|_2$) with $\lambda>0$.

\section{Preliminaries}\label{online_opsec:preliminaries}
In this section, we present some definitions, properties, and assumptions related to graph theory, projections, subgradients, and Bregman divergence.

\subsection{Graph Theory}
Interactions between agents is modeled by a time-varying directed graph. Specifically, at time $t$, agents communicate with each other according to a directed graph $\mathcal{G}_t=(\mathcal{V},\mathcal{E}_t)$, where $\mathcal{V}=[n]$ is the agent set and $\mathcal{E}_t\subseteq\mathcal{V}\times\mathcal{V}$ is the edge set. A directed edge $(j,i)\in\mathcal{E}_t$ means that agent $i$ can receive data broadcasted by agent $j$ at time $t$. Let $\mathcal{N}^{\inout}_i(\mathcal{G}_t)=\{j\in [n]\mid (j,i)\in\mathcal{E}_t\}$ and $\mathcal{N}^{\outin}_i(\mathcal{G}_t)=\{j\in [n]\mid (i,j)\in\mathcal{E}_t\}$ be the sets of in- and out-neighbors, respectively, of agent $i$ at time $t$. A directed path is a sequence of consecutive directed edges, and a graph is called strongly connected if there is at least one directed path
from any agent to any other agent in the graph. The adjacency matrix $W_t\in\mathbb{R}^{n\times n}$ at time $t$ fulfills $[W_t]_{ij}>0$ if $(j,i)\in\mathcal{E}_t$ or $i=j$, and $[W_t]_{ij}=0$ otherwise.

The following mild assumption is made on the graph.
\begin{assumption}\label{online_op:assgraph}
For any $t\in\mathbb{N}_+$, the graph $\mathcal{G}_t$ satisfies the following conditions:
\begin{enumerate}
  \item There exists a constant $w\in(0,1)$, such that $[W_t]_{ij}\ge w$ if $[W_t]_{ij}>0$.
  \item The adjacency matrix $W_t$ is doubly stochastic, i.e., $\sum_{i=1}^n[W_t]_{ij}=\sum_{j=1}^n[W_t]_{ij}=1,~\forall i,j\in[n]$.
  \item There exists an integer $\iota>0$ such that the graph $(\mathcal{V},\cup_{l=0,\dots,\iota-1}\mathcal{E}_{t+l})$ is strongly connected.
\end{enumerate}
\end{assumption}

\subsection{Projections}
For a set $\calS\subseteq\mathbb{R}^p$, $\calP_{\calS}(\cdot)$ is the projection operator
\begin{align*}
\calP_{\calS}(y)=\argmin_{x\in\calS}\|x-y\|^2,~\forall y\in\Real^{p}.
\end{align*}
This  projection always exists and is unique when $\calS$ is closed and convex \cite{boyd2004convex}. For simplicity, we use $[\cdot]_+$ to denote $\calP_{\calS}(\cdot)$ when $\calS=\mathbb{R}^n_+$, which satisfies
\begin{align}\label{online_op:proj}
\|[x]_+-[y]_+\|\le\|x-y\|,~\forall x,y\in\mathbb{R}^p.
\end{align}
Moreover, if a function $f:\Dom\rightarrow\mathbb{R}$ is convex, then $[f]_+$ is also convex.

\subsection{Subgradients}
\begin{definition}
Let $f:\Dom\rightarrow\mathbb{R}$ be a function with $\Dom\subset\mathbb{R}^p$. A vector $g\in\mathbb{R}^p$ is called a subgradient of $f$ at $x\in\Dom$ if
\begin{align}\label{online_op:subgradient}
f(y)\ge f(x)+\langle g,y-x\rangle,~\forall y\in\Dom.
\end{align}
The set of all subgradients of $f$ at $x$, denoted $\partial f(x)$, is called the subdifferential of $f$ at $x$.
\end{definition}
When the function $f$ is convex and differentiable, then its subdifferential at any point $x$ only has a single element,
which is exactly its gradient, denoted $\nabla f(x)$. With a slight abuse of the notation, we use $\nabla f(x)$ to denote the subgradient of $f$ at $x$ also when $f$ is not differentiable. Then, $\partial f(x)=\{\nabla f(x)\}$. If $f$ is a closed convex function, then $\partial f(x)$ is non-empty for any $x\in\Dom$ \cite{bubeck2015convex}.
Similarly, for a vector function $f=[f_1,\dots,f_m]^\top:\Dom\rightarrow\mathbb{R}^m$, its subgradient at $x\in\Dom$ is denoted as
\begin{align*}
\nabla f(x)=\left[\begin{array}{c}(\nabla f_1(x))^\top\\
(\nabla f_2(x))^\top\\
\vdots\\
(\nabla f_m(x))^\top
\end{array}\right]\in\mathbb{R}^{m\times p}.
\end{align*}

We make the following standing assumption on the cost, regularization, and constraint functions.
\begin{assumption}\label{online_op:assfunction}
\begin{enumerate}
  \item The set $X_i$ is convex and compact for all $i\in[n]$.
  \item $\{f_{i,t}\}$, $\{r_{i,t}\}$,  and $\{g_{i,t}\}$ are convex and uniformly bounded on $X_i$, i.e., there exists a constant $F>0$ such that
      \begin{align}
      &\|f_{i,t}(x)\|\le F,~\|r_{i,t}(x)\|\le F,\nonumber\\
      &\|g_{i,t}(x)\|\le F,~\forall t\in\mathbb{N}_+,~\forall i\in[n],~\forall x\in X_i.\label{online_op:ftgtupper}
      \end{align}
  \item $\{\nabla f_{i,t}\}$, $\{\nabla r_{i,t}\}$,  and $\{\nabla g_{i,t}\}$ exist and they are uniformly bounded on $X_i$, i.e., there exists a constant $G>0$ such that
  \begin{align}\label{online_op:subgupper}
  &\|\nabla f_{i,t}(x)\|\le G,~\|\nabla r_{i,t}(x)\|\le G,\nonumber\\
  &\|\nabla g_{i,t}(x)\|\le G,~\forall t\in\mathbb{N}_+,~\forall i\in[n],~\forall x\in X_i.
  \end{align}
\end{enumerate}
\end{assumption}

\subsection{Bregman Divergence}
Each agent $i\in[n]$ uses the Bregman divergence $\calD_{\psi_i}(x,y)$ to measure the distance between $x\in X_i$ and $y\in X_i$, where
\begin{align}\label{online_op:bregmandiv}
\calD_{\psi_i}(x,y)=\psi_i(x)-\psi_i(y)-\langle \nabla \psi_i(y),x-y\rangle,
\end{align}
and $\psi_i:X_i\rightarrow\mathbb{R}$ is a differentiable and strongly convex function with  convexity parameter $\sigma_i>0$. Then,  we have
$
\psi_i(x)\ge\psi_i(y)+\langle \nabla \psi_i(y),x-y\rangle+\frac{\sigma_i}{2}\|x-y\|^2
$.
Thus,
\begin{align}\label{online_op:eqbergman}
\calD_{\psi_i}(x,y)\ge\frac{\underline{\sigma}}{2}\|x-y\|^2,
\end{align}
where $\underline{\sigma}=\min\{\sigma_1,\dots,\sigma_n\}$.
Hence, $\calD_{\psi_i}(\cdot,y)$ is a strongly convex function with convexity parameter $\underline{\sigma}$ for all $y\in X_i$. Additionally, (\ref{online_op:bregmandiv}) implies that for all $i\in[n]$ and $x,y,z\in X_i$,
\begin{align}\label{online_op:bregmthree}
&\langle y-x,\nabla\psi_i(z)-\nabla\psi_i(y) \rangle\nonumber\\
&=\calD_{\psi_i}(x,z)-\calD_{\psi_i}(x,y)-\calD_{\psi_i}(y,z).
\end{align}

Two well-known examples of Bregman divergence are Euclidean distance $\calD_{\psi_i}(x,y)=\|x-y\|^2$ (with $X_i$ an arbitrary convex and compact set in $\mathbb{R}^{p_i}$) generated from $\psi_i(x)=\|x\|^2$, and the Kullback-Leibler (KL) divergence $\calD_{\psi_i}(x,y)=-\sum_{j=1}^px_j\log\frac{y_j}{x_j}$ between two $p_i$-dimensional standard unit vectors (with $X_i$ the $p_i$-dimensional probability simplex in $\mathbb{R}^{p_i}$) generated from $\psi_i(x)=\sum_{j=1}^p(x_j\log x_j-x_j)$.
One mild assumption on the Bregman divergence is stated as follows.

\begin{assumption}\label{online_op:assbregmanlip}
For all $i\in[n]$ and $y\in X_i$, $\calD_{\psi_i}(\cdot,y):X_i\rightarrow\mathbb{R}$ is Lipschitz, i.e., there exists a constant $K>0$ such that
\begin{align}\label{online_op:bregmalip}
|\calD_{\psi_i}(x_1,y)-\calD_{\psi_i}(x_2,y)|
\le K\|x_1-x_2\|,~\forall x_1,x_2\in X_i.
\end{align}
\end{assumption}
This assumption is satisfied when $\psi_i$ is Lipschitz on $X_i$. From Assumptions~\ref{online_op:assfunction} and \ref{online_op:assbregmanlip} it follows that
\begin{align}\label{online_op:bregmanupp}
\calD_{\psi_i}(x,y)\le d(X)K,~\forall x,y\in X_i,~\forall i\in[n],
\end{align}
where $d(X)$ is a positive constant such that
\begin{align}
\|x-y\|\le d(X),~\forall x,y\in X.\label{online_op:domainupper}
\end{align}

To end this section, we introduce a generalized definition of strong convexity.
\begin{definition}
(Definition 2 in \cite{shalev2007logarithmic}) A convex function $f:\Dom\rightarrow \mathbb{R}$ is $\mu$-strongly convex over the convex set $\Dom$ with respect to a strongly convex and differentiable function $\psi$ with $\mu>0$  if for all $ x,y\in\Dom$,
\begin{align*}
f(x)\ge f(y)+\langle x-y,\nabla f(y)\rangle+\mu\calD_{\psi}(x,y).
\end{align*}
\end{definition}
This definition generalizes the usual definition of strong convexity by replacing the Euclidean distance with the Bregman divergence.

%This assumption is satisfied when $ X_i$ is compact and $\psi_i$ is continuously differentiable on $X_i$.
%On the other hand, it should be pointed out that Assumption~\ref{online_op:assbregmanupp} and (\ref{online_op:eqbergman}) imply
%\begin{align}
%\|x-y\|\le d(X),~\forall x,y\in X,\label{online_op:domainupper}
%\end{align}
%where $d(X)=\sqrt{\frac{2nD_{\max}}{\underline{\sigma}}}$. For simplicity, if there is no boundedness assumption on the closed convex set $X$ or Assumption~\ref{online_op:assbregmanupp}, we let $d(X)=\infty$.

\section{Distributed Online Primal-Dual Dynamic Mirror Descent Algorithms}\label{online_opsec:algorithm}

In this section, we propose a distributed online primal-dual dynamic mirror descent algorithm for solving the convex optimization problem (\ref{online:problem1}). In the next section, we derive regret and constraint violation bounds for this algorithm.
\begin{algorithm}[tb]
\caption{Distributed Online Primal-Dual Dynamic Mirror Descent}
\label{online_op:algorithm}
\begin{algorithmic}[1]
\STATE \textbf{Input:}  non-increasing sequences $\{\alpha_t>0\}$, $\{\beta_t>0\}$, and $\{\gamma_t>0\}$; differentiable and strongly convex functions $\{\psi_i,~i\in[n]\}$.
\STATE \textbf{Initialize:} $x_{i,0}\in X_i$, $f_{i,0}(\cdot)=r_{i,0}(\cdot)\equiv0$, $g_{i,0}(\cdot)\equiv{\bf 0}_m$, and $q_{i,0}={\bf0}_{m},~\forall i\in[n]$.
\FOR{$t=1,\dots,T$}
\FOR{$i=1,\dots,n$}
\STATE  Observe $\nabla f_{i,t-1}(x_{i,t-1})$, $\nabla g_{i,t-1}(x_{i,t-1})$, $g_{i,t-1}(x_{i,t-1})$, and $r_{i,t-1}(\cdot)$;
\STATE Determine $\Phi_{t,i}(\cdot)$;
\STATE Receive $[W_{t-1}]_{ij}q_{j,t-1},~j\in\mathcal{N}^{\inout}_i(\mathcal{G}_{t-1})$;
\STATE  Update
\begin{align}
       \tilde{q}_{i,t}=&\sum_{j=1}^n[W_{t-1}]_{ij}q_{j,t-1},\label{online_op:al_qhat}\\
       a_{i,t}=&\nabla f_{i,t-1}(x_{i,t-1})\nonumber\\
       &+(\nabla g_{i,t-1}(x_{i,t-1}))^\top \tilde{q}_{i,t},\label{online_op:al_bigomega}\\
       \tilde{x}_{i,t}=&\argmin_{x\in X_i}\{\alpha_t\langle x,a_{i,t}\rangle+\alpha_tr_{i,t-1}(x)\nonumber\\
       &~~~~~~~~~~~~~~~+\calD_{\psi_i}(x,x_{i,t-1})\},\label{online_op:al_x}\\
       b_{i,t}=&\nabla g_{i,t-1}(x_{i,t-1})(\tilde{x}_{i,t}-x_{i,t-1})\nonumber\\
       &+g_{i,t-1}(x_{i,t-1}),\label{online_op:al_b}\\
       q_{i,t}=&[\tilde{q}_{i,t}+\gamma_t(b_{i,t}-\beta_t\tilde{q}_{i,t})]_{+},\label{online_op:al_q}\\
       x_{i,t}=&\Phi_{i,t}(\tilde{x}_{i,t});\label{online_op:al_xat}
       \end{align}
\STATE  Broadcast $q_{i,t}$ to $\mathcal{N}^{\outin}_i(\mathcal{G}_{t})$.
\ENDFOR
\ENDFOR
\STATE  \textbf{Output:} $\bsx_{T}$.
%\STATE  \textbf{Stop:} when the user wants to stop.
\end{algorithmic}
\end{algorithm}

The algorithm is given in pseudo-code as in Algorithm~\ref{online_op:algorithm}.
In this algorithm, each agent $i$ maintains two local sequences: the local primal decision variable sequence $\{x_{i,t}\}\subseteq X_i$ and the local dual variable sequence $\{q_{i,t}\}\subseteq\mathbb{R}^m_+$. They are initialized by an arbitrary $x_{i,0}\in X_i$ and $q_{i,0}={\bf0}_{m}$ and updated recursively using the update rules (\ref{online_op:al_qhat})--(\ref{online_op:al_xat}). Specifically, each agent $i$ averages its local dual variable with its in-neighbors in the consensus step (\ref{online_op:al_qhat}); computes the updating direction information for the local primal variable, $a_{i,t}$, in (\ref{online_op:al_bigomega}); updates the temporary decision $\tilde{x}_{i,t}$ through the composite objective mirror descent (\ref{online_op:al_x}); computes the updating direction information for the local dual variable, $b_{i,t}$, in (\ref{online_op:al_b}); updates the local dual variable $q_{i,t}$ in (\ref{online_op:al_q}); and updates the local decision variable $x_{i,t}$ in (\ref{online_op:al_xat}), where $\Phi_{i,t}:X_i\rightarrow X_i$ is a dynamic mapping that characterizes agent $i$'s estimate of the dynamics of the optimal sequences to problem (\ref{online:problem1}). If the agent lacks information on the optimal sequence, $\Phi_{i,t}$ is simply set to the identity mapping.

\begin{remark}
In Algorithm~\ref{online_op:algorithm}, $\{\alpha_t>0\}$ and $\{\gamma_t>0\}$ are the stepsize sequences used in the primal and dual updates, respectively, and $\{\beta_t>0\}$ are the regularization parameters (for simplicity called stepsizes as well). These sequences play a key role in deriving the regret and constraint violation bounds. They allow the trade-off between how fast these two bounds tend to zero.
This is in contrast to most algorithms, which typically use the same stepsizes for the primal and dual updates. Different stepsizes have also been used in \cite{jenatton2016adaptive,yuan2017adaptive}.
The penalty term $-\beta_t\tilde{q}_{i,t}$ in (\ref{online_op:al_q}) is used to prevent the dual variable growing too large. A penalty term is commonly used when transforming constrained to unconstrained problems \cite{mahdavi2012trading,jenatton2016adaptive,sun2017safety,yuan2017adaptive,Li2018distributed}.
With some modifications, all the results in this paper still hold if the coordinated sequences $\alpha_t,~\beta_t,~\gamma_t$ are replaced by uncoordinated ones $\alpha_{i,t},~\beta_{i,t},~\gamma_{i,t}$.
\end{remark}

\begin{remark}
At time $t$, each agent $i$ needs to know the regularization function at the previous time $t-1$, i.e., $r_{i,t-1}(\cdot)$. This is in many situations a mild assumption since regularization functions are normally predefined to influence the structure of the decision. Furthermore,  $g_{i,t-1}(x_{i,t-1})$, $\nabla f_{i,t-1}(x_{i,t-1})$, and $\nabla g_{i,t-1}(x_{i,t-1})$ rather than the full knowledge of $f_{i,t-1}(\cdot)$ and $g_{i,t-1}(\cdot)$ are needed, similar to the assumption on most online algorithms in the literature, cf., \cite{mahdavi2012trading,jenatton2016adaptive,sun2017safety,NIPS2018_7852,Li2018distributed}. Note that the total number of iterations or any parameters related to the objective or constraint functions, such as upper bounds of the objective and constraint functions or their subgradients, are not used in the algorithm. Also note that no local information related to the primal is exchanged between the agents, but only local dual variables.
\end{remark}

\begin{remark}
The composite objective mirror descent (\ref{online_op:al_x}) is  almost the same as the mirror descent, but with the important difference that the regularization function is not linearized. The regularization function can often lead to sparse updates \cite{duchi2010composite}.
The minimization problem (\ref{online_op:al_x}) is strongly convex, so it is solvable  at a linear convergence rate and closed-form solutions are available in special cases. For example, if $r_{i,t}$ is a constant mapping and Euclidean distance is used as the Bregman distance,  i.e., $\calD_{\psi_i}(x,y)=\|x-y\|^2$, then (\ref{online_op:al_x}) can be solved by the projection
$
\tilde{x}_{i,t}=\calP_{X_i}(x_{i,t-1}-\frac{\alpha_t}{2}a_{i,t})
$.
\end{remark}

\begin{remark}
If the optimal sequence of agent $i$ has the dynamics
$
x_{i,t}^*=\Phi^*_{i,t}(x_{i,t-1}^*)
$ for some true dynamic mapping $\Phi^*_{i,t}: X_i\rightarrow X_i$, then $\Phi_{i,t}$ can be viewed as an estimate of $\Phi^*_{i,t}$. If $\Phi_{i,t}$ is equal or close enough to $\Phi^*_{i,t}$, then $x_{i,t}^*-\Phi_{i,t}(x_{i,t-1}^*)=\Phi^*_{i,t}(x_{i,t-1}^*)-\Phi_{i,t}(x_{i,t-1}^*)$ is small. Actually, $\Phi_{i,t}$ is a decentralized variant of the dynamical model $\Phi_t$ introduced in \cite{hall2015online}. $\Phi_{i,t}$ is chosen as the identity mapping if at time $t$ agent $i$ has no knowledge about the dynamics of the optimal sequence.
\end{remark}

To end this section, an assumption on the  dynamic mapping $\Phi_{i,t}$ is introduced.
\begin{assumption}\label{online_op:assnonexpansive}
For any $t\in\mathbb{N}_+$ and $i\in[n]$, the dynamic mapping $\Phi_{i,t}$ is contractive, i.e.,
\begin{align}\label{online_op:assnonexpansiveequ}
\calD_{\psi_i}(\Phi_{i,t}(x),\Phi_{i,t}(y))\le \calD_{\psi_i}(x,y),~\forall x,y\in X_i.
\end{align}
\end{assumption}

\section{Regret and Constraint Violation Bounds}\label{online_opsec:main}
This section presents the main results on  regret and constraint violation bounds for Algorithm~\ref{online_op:algorithm}, but first some preliminary results are given.

\subsection{Preliminary Results}
Firstly, we present two results on the regularized Bregman projection.
\begin{lemma}\label{online_op:lemma_mirror}
Suppose that $\psi:\mathbb{R}^p\rightarrow\mathbb{R}^p$ is a strongly convex function with convexity parameter $\sigma>0$ and  $h:\Dom\rightarrow\Dom$ is a convex function with $\Dom$ being a convex and closed set in $\mathbb{R}^p$. Moreover, assume that $\nabla h(x),~\forall x\in\Dom$, exists and there exists $G_h>0$ such that $\|\nabla h(x)\|\le G_h,~\forall x\in\Dom$. Given $z\in\Dom$,
the regularized Bregman projection
\begin{align}\label{online_op:lemma_mirroreq1}
y=\argmin_{x\in\Dom}\{h(x)+\calD_\psi(x,z)\},
\end{align}
satisfies the following inequalities
\begin{align}
\langle y-x,\nabla h(y)\rangle\le& \calD_\psi(x,z)-\calD_\psi(x,y)\nonumber\\
&-\calD_\psi(y,z),~\forall x\in\Dom,\label{online_op:lemma_mirroreq2}\\
\|y-z\|\le&\frac{G_h}{\sigma}.\label{online_op:lemma_mirrorine}
\end{align}
\end{lemma}
\begin{proof}
See Appendix~\ref{online_op:lemma_mirrorproof}.
\end{proof}
%{\color{blue} \begin{remark}
%(\ref{online_op:lemma_mirroreq2}) is a generalization of Lemma 6 in \cite{shahrampour2018distributed} in the sense that the strongly convex parameter $\sigma$ could be any positive real numbers.
%\end{remark}}
\begin{remark}
Note that (\ref{online_op:lemma_mirroreq2}) extends Lemma 6 in \cite{shahrampour2018distributed} and (\ref{online_op:lemma_mirrorine}) presents an upper bound on the deviation of the optimal point from a fixed point for the regularized Bregman projection.
\end{remark}

Next we state some results on the local dual variables.
\begin{lemma}\label{online_op:lemma_virtualbound}
Suppose Assumptions \ref{online_op:assgraph}--\ref{online_op:assfunction} hold. For all $i\in[n]$ and $t\in\mathbb{N}_+$, the $\tilde{q}_{i,t}$ and $q_{i,t}$ generated by Algorithm \ref{online_op:algorithm} satisfy
\begin{align}
&\|\tilde{q}_{i,t}\|\le \frac{F}{\beta_t},~\|q_{i,t}\|\le \frac{F}{\beta_t},\label{online_op:lemma_virtualboundeqy}\\
&\|\tilde{q}_{i,t+1}-\bar{q}_{t}\|
\le n\tau B_1\sum_{s=1}^{t-1}\gamma_{s+1}\lambda^{t-1-s},\label{online_op:lemma_qbarequ}\\
&\frac{\Delta_{t}}{2\gamma_t}\le \frac{n(B_1)^2}{2}\gamma_{t}
+[\bar{q}_{t-1}-q]^\top g_{t-1}(x_{t-1})
+E_{1}(t)\nonumber\\
&~~~~~~~~+\frac{\underline{\sigma}}{4\alpha_t}\sum_{i=1}^n\|\tilde{x}_{i,t}-x_{i,t-1}\|^2
+n\Big(\frac{G^2\alpha_t}{\underline{\sigma}}+\frac{\beta_t}{2}\Big)\|q\|^2\nonumber\\
&~~~~~~~~+\sum_{i=1}^n[\tilde{q}_{i,t}]^\top\nabla g_{i,t-1}(x_{i,t-1})(\tilde{x}_{i,t}-x_{i,t-1})
,\label{online_op:gvirtualnorm}
\end{align}
where $\bar{q}_{t}=\frac{1}{n}\sum_{i=1}^nq_{i,t}$,
$\tau=(1-w/2n^2)^{-2}>1$, $B_1=2F+Gd(X)$, $\lambda=(1-w/2n^2)^{1/\iota}$,
$$\Delta_{t}=\sum_{i=1}^n\|q_{i,t}-q\|^2
-(1-\beta_t\gamma_t)\sum_{i=1}^n\|q_{i,t-1}-q\|^2,$$
$q$ is an arbitrary vector in $\mathbb{R}^m_+$,
and $$E_{1}(t)=n^2\tau B_1F\sum_{s=1}^{t-1}\gamma_{s+1}\lambda^{t-1-s}.$$
\end{lemma}
\begin{proof}
See Appendix~\ref{online_op:lemma_virtualboundproof}.
\end{proof}

\begin{remark}
With the help of the penalty term $-\beta_t\tilde{q}_{i,t}$, (\ref{online_op:lemma_virtualboundeqy}) gives an upper bound of the local dual variables even without Slater's condition. (\ref{online_op:lemma_qbarequ}) is a standard estimate from the consensus protocol with perturbations and time-varying communication graphs \cite{lee2017sublinear} and presents an upper bound on the deviation of the local estimate from the average value of the local dual variables at each iteration. (\ref{online_op:gvirtualnorm}) gives an upper bound on the regularized drift of the local dual variables  $\Delta_t$, which extends Lemma~3 in \cite{hall2015online} from a centralized setting to a distributed one.
\end{remark}

Next,  we provide an upper bound on the regret for one update step.
\begin{lemma}\label{online_op:lemma_regretdelta}
Suppose Assumptions~\ref{online_op:assgraph}--\ref{online_op:assnonexpansive} hold. For all $i\in[n]$, let $\{x_{t}\}$ be the sequence generated by Algorithm \ref{online_op:algorithm} and $\{y_t\}$ be an arbitrary sequence in $X$, then
\begin{align}\label{online_op:lemma_regretdeltaequ}
[\bar{q}_t&]^\top g_{t}(x_{t})+f_{t}(x_{t})-f_{t}(y_{t})\nonumber\\
\le&[\bar{q}_{t}]^\top g_{t}(y_{t})+2E_{1}(t+1)+E_{2}(t+1)\nonumber\\
&+\frac{4nG^2\alpha_{t+1}}{\underline{\sigma}}
+\frac{K}{\alpha_{t+1}}\sum_{i=1}^n\|y_{i,t+1}-\Phi_{i,t+1}(y_{i,t})\|\nonumber\\
&-\sum_{i=1}^n[\tilde{q}_{i,t+1}]^\top\nabla g_{i,t}(x_{i,t})(\tilde{x}_{i,t+1}-x_{i,t})\nonumber\\
&-\frac{\underline{\sigma}}{4\alpha_{t+1}}\sum_{i=1}^n\|\tilde{x}_{i,t+1}-x_{i,t}\|^2,~\forall t\in\mathbb{N}_+,
\end{align}
where
\begin{align*}
E_{2}(t)=\frac{1}{\alpha_{t}}\sum_{i=1}^n[\calD_{\psi_i}(y_{i,t-1},x_{i,t-1})
-\calD_{\psi_i}(y_{i,t},x_{i,t})].
\end{align*}
\end{lemma}
\begin{proof}
See Appendix~\ref{online_op:lemma_regretdeltaproof}.
\end{proof}

Finally, we derive regret and constraint violation bounds for Algorithm~\ref{online_op:algorithm}.

\begin{lemma}\label{online_op:theoremreg}
Suppose Assumptions~\ref{online_op:assgraph}--\ref{online_op:assnonexpansive} hold. For any $T\in\mathbb{N}_+$, let $\bsx_T$ be the sequence generated by Algorithm~\ref{online_op:algorithm}. Then, for any comparator sequence $\bsy_T\in\calX_{T}$,
\begin{align}\label{online_op:theoremregequ}
&\Reg(\bsx_T,\bsy_T)\nonumber\\
&\le C_{1,1}\sum_{t=1}^T\gamma_{t+1}+C_{1,2}\sum_{t=1}^T\alpha_{t+1}+\sum_{t=1}^TE_{2}(t+1)
\nonumber\\
&~~~-\frac{1}{2}\sum_{t=1}^T\sum_{i=1}^n\Big[\frac{1}{\gamma_{t}}-\frac{1}{\gamma_{t+1}}
+\beta_{t+1}\Big]\|q_{i,t}\|^2+\frac{KV_{\Phi}(\bsy_T)}{\alpha_{T}},
\end{align}
and
\begin{align}\label{online_op:theoremconsequ}
&\|[\sum_{t=1}^Tg_{t}(x_{t})]_+\|^2\nonumber\\
&\le4n\Big[\frac{1}{\gamma_1}
+\sum_{t=1}^T\Big(\frac{G^2\alpha_{t+1}}{\underline{\sigma}}+\frac{\beta_{t+1}}{2}\Big)\Big]\bigg\{2nFT
+\frac{KV^*_{\Phi}}{\alpha_{T}}\nonumber\\
&~~~+C_{1,1}\sum_{t=1}^T\gamma_{t+1}
+C_{1,2}\sum_{t=1}^T\alpha_{t+1}+\sum_{t=1}^TE_{2}(t+1)
\nonumber\\
&~~~-\frac{1}{2}\sum_{t=1}^T\sum_{i=1}^n\Big(\frac{1}{\gamma_{t}}
-\frac{1}{\gamma_{t+1}}+\beta_{t+1}\Big)\|q_{i,t}-q_0\|^2\bigg\},
\end{align}
where $C_{1,1}=\frac{3n^{2}\tau B_1F}{1-\lambda}+\frac{n(B_1)^2}{2}$, $C_{1,2}=\frac{4nG^2}{\underline{\sigma}}$ are constants independent of $T$,
\begin{align*}
V_{\Phi}(\bsy_T)=\sum_{t=1}^{T-1}\sum_{i=1}^n\|y_{i,t+1}-\Phi_{i,t+1}(y_{i,t})\|
\end{align*}
is the accumulated dynamic variation of the sequence $\bsy_T$ with respect to $\{\Phi_{i,t}\}$,
\begin{align*}
V^*_{\Phi}=\min_{\bsy_T\in\calX_{T}}V_{\Phi}(\bsy_T)
\end{align*}
is the minimum accumulated dynamic variation of all feasible sequences, and $$q_0=\frac{[\sum_{t=1}^Tg_{t}(x_{t})]_+}
{2n[\frac{1}{\gamma_1}
+\sum_{t=1}^T(\frac{G^2\alpha_{t+1}}{\underline{\sigma}}+\frac{\beta_{t+1}}{2})]}.$$
\end{lemma}
\begin{proof}
See Appendix~\ref{online_op:theoremregproof}.
\end{proof}

%\begin{proposition}\label{online_op:theoremcons}
%Under the same conditions as stated in Proposition~\ref{online_op:theoremreg}, then,
%\end{proposition}
%\begin{proof}
%See Appendix~\ref{online_op:theoremconsproof}.
%\end{proof}

\begin{remark}
Note that
the dependence on the stepsize sequences, the accumulated dynamic variation of the comparator sequence,  the number of agents, and the network connectivity is characterized in the regret and constraint violation bounds above. The accumulated variation of constraints or the point-wise maximum variation of consecutive constraints defined in \cite{chen2017online} do, however, not appear in these bounds. This regret bound is the same as the regret bound achieved by the centralized dynamic mirror descent in \cite{hall2015online}, while \cite{hall2015online} only considered static set constraints.
\end{remark}

\begin{remark}
The factor $V^*_{\Phi}$ in (\ref{online_op:theoremconsequ}) can be replaced by $V_\Phi(\bsy_T)$ since $V^*_{\Phi}\le V_\Phi(\bsy_T)$. Moreover, if all $\{\Phi_{t,i}\}$ are the identity mapping, then $V^*_{\Phi}=\min_{\bsy_T\in\check{\calX}_{T}}V_{\Phi}(\bsy_T)=V_{\Phi}(\check{\bsx}_T^*)=0$.
\end{remark}

\subsection{Dynamic Regret and Constraint Violation Bounds}
This section states the  main results on dynamic regret and  constraint violation bounds for Algorithm~\ref{online_op:algorithm}. The succeeding theorem characterizes the bounds based on some natural decreasing stepsize sequences.

%For some sequences of step-size $\alpha_t$ and $\beta_t$, we could show that the regret $\Reg_T(\bsy_T)$ grows sub-linearly with respect to $T$.
\begin{theorem}\label{online_op:corollaryreg}
Suppose Assumptions~\ref{online_op:assgraph}--\ref{online_op:assnonexpansive} hold. For any $T\in\mathbb{N}_+$, let $\bsx_T$ be the sequence generated by Algorithm~\ref{online_op:algorithm} with
\begin{align}\label{online_op:stepsize1}
\alpha_t=\frac{1}{t^{c}},~\beta_t=\frac{1}{t^\kappa},
~\gamma_t=\frac{1}{t^{1-\kappa}},~\forall t\in\mathbb{N}_+,
\end{align} where $\kappa\in(0,1)$ and $c\in(0,1)$ are constants. Then,
\begin{align}
& \Reg(\bsx_T,\bsx^*_T)\le C_1T^{\max\{1-c,c,\kappa\}}+2KT^{c}V_{\Phi}(\bsx^*_T),\label{online_op:corollaryregequ1}\\
&\|[\sum_{t=1}^Tg_{t}(x_{t})]_+\|^2\le C_{2}T^{\max\{2-c,2-\kappa\}}\nonumber\\
&~~~~~~~~~~~~~~~~~~~~~~~+KC_{2,1}T^{\max\{1,1+c-\kappa\}}V^*_{\Phi},\label{online_op:corollaryconsequ}
\end{align}
where $C_1=\frac{C_{1,1}}{\kappa}+\frac{C_{1,2}}{1-c}+2nd(X)K$, $C_{2}=C_{2,1}(2nF+C_1)$, and $C_{2,1}=2n(\frac{2G^2}{(1-c)\underline{\sigma}}+\frac{1}{1-\kappa}+2)$ are constants independent of $T$.
\end{theorem}
\begin{proof}
See Appendix~\ref{online_op:corollaryregproof}.
\end{proof}

%For some sequences of step-size $\alpha_t$ and $\beta_t$, we could show that the constraint violation $\Reg^c_T$ grows sub-linearly with respect to $T$.

%\subsection{Discussions}
%Now let us discuss Algorithm~\ref{online_op:algorithm} with the step-size sequences in (\ref{online_op:stepsize1}) and results shown in Theorem~\ref{online_op:corollaryreg}.

\begin{remark}\label{online_op:remarksub}
Sublinear  dynamic regret and constraint violation is thus achieved if $V_{\Phi}(\bsx^*_T)$ grows sublinearly. If, in this case, there exists a constant $\nu\in[0,1)$, such that $V_{\Phi}(\bsx^*_T)=\mathcal{O}(T^\nu)$, then setting $c\in(0,1-\nu)$ in Theorem~\ref{online_op:corollaryreg} gives $\Reg(\bsx_T,\bsx^*_T)=\mathbf{o}(T)$ and $\|[\sum_{t=1}^Tg_{t}(x_{t})]_+\|=\mathbf{o}(T)$.
\end{remark}

\begin{remark}
$V_\Phi(\bsx^*_T)$ depends on the dynamic mapping $\Phi_{i,t}$. In practice, agents may not know what is a good estimate of $\Phi_{i,t}$ and $\Phi_{i,t}$ may change stochastically. It is for future research how to estimate $\Phi_{i,t}$ from a finite or parametric class of candidates.
\end{remark}
%Theorems~\ref{online_op:theoremstatic} and \ref{online_op:theoremstongconvex} show that the constraint violation bound is at least $\mathcal{O}(T^{\frac{3}{4}})$ while the static regret bound could be $\mathcal{O}(T^{\frac{1}{2}})$.

From (\ref{online_op:corollaryconsequ}), we can see that the constraint violation bound is strictly greater than $\mathcal{O}(\sqrt{T})$ since $\max\{2-c,2-\kappa\}>1$.
In the following we show that an $\mathcal{O}(\sqrt{T})$ bound on constraint violation can be achieved if all $\{\Phi_{i,t}\}$ are the identity mapping and the constraint functions $\{g_{i,t}\}$ satisfy Slater's condition, which was assumed in \cite{chen2017online,neely2017online}.
\begin{assumption}\label{online_op:assgt}
(Slater's condition) There exists a constant $\varepsilon>0$ and a vector $x_0\in X$, such that
\begin{align}\label{online_op:gtcon}
g_{t}(x_0)\le-\varepsilon{\bf 1}_{m},~t\in\mathbb{N}_+.
\end{align}
\end{assumption}

\begin{theorem}\label{online_op:theoremslater}
Suppose Assumptions~\ref{online_op:assgraph}--\ref{online_op:assgt} hold. For any $T\in\mathbb{N}_+$, let $\bsx_T$ be the sequence generated by Algorithm~\ref{online_op:algorithm} with all $\{\Phi_{t,i}\}$ being the identity mapping, and
\begin{align}
\alpha_t=\frac{1}{t^{1-\kappa}},~\beta_t=\frac{1}{t^{\kappa}},
~\gamma_t=\frac{1}{t^{1-\kappa}},~\forall t\in\mathbb{N}_+\label{online_op:stepsizeslater},
\end{align}  where $\kappa\in(0,1)$. Then,
\begin{align}
&\Reg(\bsx_T,\bsx^*_T)\le  C_1T^{\max\{1-\kappa,\kappa\}}
+2KT^{1-\kappa}V_I(x^*_T),\label{online_op:regslater}\\
&\|[\sum_{t=1}^Tg_{t}(x_{t})]_+\|\le C_{3}T^{\max\{1-\kappa,\kappa\}},\label{online_op:regcslater}
\end{align}
where $V_I(x^*_T)=\sum_{t=1}^{T-1}\|x^*_{t+1}-x^*_t\|$ is the accumulated variation of the optimal sequence $\bsx^*_T$, $C_3=n[2B_2+\frac{B_2}{1-\kappa}+\frac{G^2(B_2+2)\sqrt{m}}{\underline{\sigma}\kappa}]$, $B_2=\max\{2\varepsilon+2\sqrt{\varepsilon^2+nd(X)K},\frac{2B_3}{\varepsilon}\}$, and $B_3=2F+C_{1,1}$ are constants independent of $T$.
\end{theorem}
\begin{proof} See Appendix~\ref{online_op:theoremslaterproof}.
\end{proof}

\begin{remark}
From (\ref{online_op:regcslater}), we note that under Slater's condition the constraint violation bound is not affected by the optimal sequences or the point-wise maximum variation of consecutive constraints, which is different from the bounds obtained in \cite{chen2017online}.
From (\ref{online_op:regslater}), it follows, similarly to Remark \ref{online_op:remarksub}, that sublinear  dynamic regret could be achieved if $V_I(\bsx^*_T)$ grows sublinearly with a known upper bound. Then, there exists a constant $\nu\in[0,1)$, such that $V_I(\bsx^*_T)=\mathcal{O}(T^\nu)$, so setting $\kappa\in(\nu,1)$ in Theorem~\ref{online_op:theoremslater} gives $\Reg(\bsx_T,\bsx^*_T)=\mathbf{o}(T)$ and $\|[\sum_{t=1}^Tg_{t}(x_{t})]_+\|=\mathbf{o}(T)$. Under the additional assumption that the accumulated variation of constraints grows sublinearly with a known upper bound, similar results have been achieved by the modified centralized online saddle-point method proposed in \cite{chen2017online}. However, \cite{chen2017online} assumed not only that the time-varying constraint functions satisfy Slater's condition but also that the slack constant is larger than the point-wise maximum variation of consecutive constraints.  The latter assumption is not always satisfied. Moreover, in \cite{chen2017online} the total number of iterations $T$ needs to be known in advance.
\end{remark}

\subsection{Static Regret and Constraint Violation Bounds}
This section states the  main results on static regret and  constraint violation bounds for Algorithm~\ref{online_op:algorithm}.
When considering static regret, $\{\Phi_{i,t}\}$ should be set to the identity mapping since the static optimal sequence is used as the comparator sequence. In this case, replacing $\bsx^*_T$ by the static sequence $\check{\bsx}_T^*$ in Theorem~\ref{online_op:corollaryreg}  gives %$V_\Phi(\bsx^*_T)=M_\Phi(\{g_t\}_{t=1}^T)=V_I(\bsx_T^*)=0$ and $\Reg(\bsx_T,\bsx^*_T)=\Reg(\bsx_T,\bar{\bsx}_T^*)$. Thus, we have
the following results on the bounds of static regret and constraint violation.
\begin{corollary}\label{online_op:theoremstatic}
Under the same conditions as stated in Theorem~\ref{online_op:corollaryreg} with all $\{\Phi_{i,t}\}$ being the identity mapping and $c=\kappa$, it holds that
\begin{align}
&\Reg(\bsx_T,\check{\bsx}_T^*)\le  C_1T^{\max\{1-\kappa,\kappa\}},\label{online_op:staticregequ1}\\
&\|[\sum_{t=1}^Tg_{t}(x_{t})]_+\|\le \sqrt{C_{2}}T^{1-\kappa/2}.\label{online_op:staticconsequ}
\end{align}
\end{corollary}
\begin{proof}
Substituting $c=\kappa$ in Theorem~\ref{online_op:corollaryreg} gives the results.
\end{proof}

\begin{remark}
From Corollary~\ref{online_op:theoremstatic}, we know that Algorithm~\ref{online_op:algorithm} achieves the same static regret and constraint violation bounds as in \cite{jenatton2016adaptive}. As discussed in \cite{jenatton2016adaptive}, $\kappa\in(0,1)$ is a user-defined trade-off parameter which enables the user to trade-off the static regret bound for the constraint violation bound. Corollary~\ref{online_op:theoremstatic} recovers the $\mathcal{O}(\sqrt{T})$ static regret bound and $\mathcal{O}(T^{3/4})$ constraint violation bound from \cite{mahdavi2012trading,sun2017safety} when $\kappa=0.5$. Moreover, the result extends the $\mathcal{O}(T^{2/3})$ bound for both static regret and constraint violation achieved in \cite{mahdavi2012trading} for linear constraint functions. However, the algorithms proposed in \cite{mahdavi2012trading,jenatton2016adaptive,sun2017safety} are centralized and the constraint functions considered in \cite{mahdavi2012trading,jenatton2016adaptive} are time-invariant. Moreover, in \cite{mahdavi2012trading,sun2017safety} the total number of iterations and in \cite{mahdavi2012trading,jenatton2016adaptive,sun2017safety} the upper bounds of the objective and constraint functions and their subgradients need to be known in advance to choose the stepsize sequences.
Furthermore, Corollary~\ref{online_op:theoremstatic} achieves smaller static regret and constraint violation bounds than \cite{Li2018distributed}, although \cite{Li2018distributed} considered time-invariant coupled inequality constraints. However, \cite{Li2018distributed} did not require the time-varying directed graph to be balanced. Although the algorithm proposed in \cite{NIPS2018_7852} achieved more strict constraint violation bound than our Algorithm~\ref{online_op:algorithm}, that algorithm assumed time-invariant constraint functions and the centralized computations.
\end{remark}

\begin{corollary}\label{online_op:corollarystaticslater}
Under the same conditions as stated in Theorem~\ref{online_op:theoremslater}, it holds that
\begin{align}
\Reg(\bsx_T,\check{\bsx}_T^*)\le&  C_1T^{\max\{1-\kappa,\kappa\}},\label{online_op:staticregequ1slater}\\
\|[\sum_{t=1}^Tg_{t}(x_{t})]_+\|\le&C_{3}T^{\max\{1-\kappa,\kappa\}}.\label{online_op:staticconsequslater}
\end{align}
\end{corollary}
%\begin{proof}
%Substituting $\kappa=0.5$ in Theorem~\ref{online_op:theoremslater} gives the results.
%\end{proof}

\begin{remark}
Setting $\kappa=0.5$ in Corollary~\ref{online_op:corollarystaticslater} gives $\Reg(\bsx_T,\check{\bsx}_T^*)=\mathcal{O}(\sqrt{T})$ and $\|[\sum_{t=1}^Tg_{t}(x_{t})]_+\|=\mathcal{O}(\sqrt{T})$. Hence,  Algorithm~\ref{online_op:algorithm} achieves stronger results than \cite{neely2017online} and the same  results as \cite{yu2017online,lee2017sublinear}. However, the algorithms proposed in \cite{yu2017online,neely2017online} are centralized and in \cite{yu2017online} it is assumed that the constraint functions are independent and identically distributed. Moreover, in \cite{lee2017sublinear} the coupled inequality constraints are time-invariant and the boundedness of the dual variable sequence {\color{blue}generated by the proposed algorithm} is  explicitly assumed. %The conditions to guarantee this assumption holds are not so obvious since the dual variable sequence is generated by the designed algorithm. In this paper, it is under Slater's condition and the selected stepsize sequences in (\ref{online_op:stepsizeslater}) we show that the dual variable sequence is bounded (see the proof of Theorem~\ref{online_op:theoremslater}).
\end{remark}

The static regret bounds in Corollaries~\ref{online_op:theoremstatic} and \ref{online_op:corollarystaticslater} can be reduced, if a generalized strong convexity of the local objective functions $f_{i,t}+r_{i,t}$ is assumed.
We put the strong convexity assumption on the local cost functions $f_{i,t}$ so $r_{i,t}$ can be simply convex, such as an $\ell_1$-regularization.
\begin{assumption}\label{online_op:assstrongconvex}
For any $i\in[n]$ and $t\in\mathbb{N}_+$, $\{f_{i,t}\}$ are $\mu_i$-strongly convex over $X_i$ with respect to $\psi_i$ with $\mu_i>0$.%, i.e., for all $x,y\in\calX_i$ and $t\in\mathbb{N}_+$,
%\begin{align}
%f_{t,i}(x)\ge& f_{t,i}(y)+\langle x-y,\nabla f_{t,i}(y)\rangle+\mu_i\calD_{\psi_i}(x,y).
%\end{align}
\end{assumption}
\begin{theorem}\label{online_op:theoremstongconvex}
Suppose Assumptions~\ref{online_op:assgraph}--\ref{online_op:assstrongconvex} hold. For any $T\in\mathbb{N}_+$, let $\bsx_T$ be the sequence generated by Algorithm~\ref{online_op:algorithm} with
\begin{align}\label{online_op:stepsizestrong}
\alpha_t=\frac{1}{t^{\max\{1-\kappa,\kappa\}}},~\beta_t=\frac{1}{t^\kappa},
~\gamma_t=\frac{1}{t^{1-\kappa}},~\forall t\in\mathbb{N}_+,
\end{align} where $\kappa\in(0,1)$. Then,
\begin{align}
&\Reg(\bsx_T,\check{\bsx}_T^*)\le \max\{C_1,C_{4}\}T^{\kappa},\label{online_op:regstongconvex}\\
&\|[\sum_{t=1}^Tg_{t}(x_{t})]_+\|\le  \sqrt{C_{2}}T^{1-\kappa/2},\label{online_op:regcstongconvex}
\end{align}
where $C_4=\frac{n(B_1)^2}{2\kappa}+\frac{B_1C_{1,1}}{\kappa}
+\frac{C_{1,2}}{\kappa}+2nd(X)K(B_4)^{1-\kappa}$, $B_4=\lceil\frac{1}{(\underline{\mu})^{\frac{1}{\kappa}}}\rceil$, and $\underline{\mu}=\min\{\mu_1,\dots,\mu_n\}$ are constants independent of $T$.
\end{theorem}
\begin{proof}
See Appendix~\ref{online_op:theoremstongconvexproof}.
\end{proof}

\begin{corollary}\label{online_op:theoremstongconvexslater}
Under the same conditions as stated in Theorem~\ref{online_op:theoremslater}, if Assumption~\ref{online_op:assstrongconvex} also holds. Then,
\begin{align}
\Reg(\bsx_T,\check{\bsx}_T^*)\le& C_{4}T^{\kappa},\label{online_op:regstongconvexslater}\\
\|[\sum_{t=1}^Tg_{t}(x_{t})]_+\|\le& C_{3}T^{\max\{1-\kappa,\kappa\}}.\label{online_op:regcstongconvexslater}
\end{align}
\end{corollary}
\begin{proof}
(\ref{online_op:regstongconvexslater}) follows from the first step in the proof of (\ref{online_op:regstongconvex}) and (\ref{online_op:regcstongconvexslater}) follows from (\ref{online_op:regcslater}).
\end{proof}
%See Appendix~\ref{online_op:theoremstongconvexslaterproof}.
\begin{remark}
With some minor modifications, the results stated in Theorem~\ref{online_op:theoremstongconvex} and Corollary~\ref{online_op:theoremstongconvexslater} still hold if Assumption~\ref{online_op:assstrongconvex} is replaced by the assumption that for any $i\in[n]$ and $t\in\mathbb{N}_+$, $f_{i,t}$ or $r_{i,t}$ is $\mu_i$-strongly convex over $X_i$ with respect to $\psi_i$ with $\mu_i>0$.
\end{remark}

%\begin{remark}
%The potential weakness of Theorem~\ref{online_op:theoremstongconvex} is that the global convexity parameter $\underline{\mu}$ needs to be known in advance. This could be partially solved by replacing the coordinated stepsize $\alpha_t$ with the uncoordinated stepsize $\alpha_{t,i}=\frac{1}{\mu_it}$ in (\ref{online_op:al_x}) and in this case, with some modifications the results in Theorem~\ref{online_op:theoremstongconvex} still hold, although each agent needs to know its own objective function's convexity parameter. The strengths are that neither $\{\calD_{\psi_i}(\cdot,\cdot)\}$ nor $\{X_i\}$ are required to be bounded and the total number of iterations is not needed.
%\end{remark}

%\section{Finite constraint violations}

%\section{Gradient-Free}
%Bandit feedback

%\section{Stochastic objective and constraint functions}
%Objective and constraint functions are generated by random processes, such as Gaussian process.

\section{NUMERICAL SIMULATIONS}\label{online_opsec:simulation}
This section evaluates the performance of Algorithm~\ref{online_op:algorithm} in solving the multi-target tracking problem introduced in Section~\ref{online_op:example}.
In the simulations, for each agent $i\in[n]$, $\Phi_{i,t}$ is set as the identity mapping and the strongly convex function $\psi_i(x)=\sigma\|x\|^2$ is used to define the Bregman divergence $\calD_{\psi_i}$. Thus, $\calD_{\psi_i}(x,y)=\sigma\|x-y\|^2, \forall i\in[n]$. The stepsize sequences given (\ref{online_op:stepsizestrong}) are used. Moreover, agent $i$ could use a regularization function  $r_{i,t}(x_{i,t})=\lambda_{i,1}\|x_{i,t}\|_1+\lambda_{i,2}\|x_{i,t}\|^2$ to influence the structure of its action, where $\lambda_{i,1}$ and $\lambda_{i,2}$ are nonnegative constants. At each time $t$, an undirected graph is used as the communication graph. Specifically, connections between vertices are random and the probability of two vertices being connected is $\rho$. To guarantee that Assumption~\ref{online_op:assgraph} holds, edges $(i,i+1),~i\in[n-1]$ are added and $[W_t]_{ij}=\frac{1}{n}$ if $(j,i)\in\mathcal{E}_t$ and $[W_t]_{ii}=1-\sum_{j\in\mathcal{N}^{\inout}_i(\mathcal{G}_t)}[W_t]_{ij}$.

We assume $n=50$, $m=5$, $\sigma=10$, $p_i=6$, $X_i=[0,5]^{p_i}$, $\zeta_{i,1}=\lambda_{i,1}=1$, $\zeta_{i,2}=\lambda_{i,2}=30$, $i\in[n]$, and $\rho=0.2$. Each component of $\pi_{i,t}$ is drawn from the discrete uniform distribution in $[0,10]$ and each component of $D_{i,t}$ is drawn from the discrete uniform distribution in $[-5,5]$. We let $y_{i,t}=[2(\zeta_{i,2}+\lambda_{i,2})x^0_{i,t}+\zeta_{i,1}\pi_{i,t}+\lambda_{i,1}{\bf 1}_{p_i}]/(2\zeta_{i,2})$, where $x^0_{i,t+1}=A_{i,t}x^0_{i,t}$ with $A_{i,t}$ being a  doubly stochastic matrix and $x^0_{i,1}$ being a vector that is uniformly drawn from $X_i$. In order to guarantee the constraints are feasible, we let $d_{i,t}=D_{i,t}x^0_{i,t}$.%, where each component of $\varpi_{i,t}$ is drawn from the discrete uniform distribution in $[1,10]$. %In order to guarantee  Assumption~\ref{online_op:assgt} holds, we uniformly chose a vector $\varpi_{i}$ from $X_i$ and set $d_{t,i}=D_{t,i} \varpi_{i}\in\mathbb{R}^{m}$.

\subsection{Dynamics of Optimal Sequences}
Under the above settings, we have that $x^*_{i,t}=x^0_{i,t}$. To investigate the dependence of the dynamic regret and constraint violation with $\Phi_{i,t}$, we run Algorithm~\ref{online_op:algorithm} for two cases: $\Phi_{i,t}$ is the identity mapping and the linear mapping $A_{i,t}$. Figs. \ref{online:figPhi} (a) and (b) show the evolutions of $\Reg(\bsx_T,\bsx^*_T)/T$ and $\|[\sum_{t=1}^Tg_{t}(x_{t})]_+\|/T$, respectively, and we can see that knowing the dynamics of the optimal sequence leads to smaller dynamic regret and  constraint violation.

\begin{figure}
\begin{subfigure}{.5\textwidth}
  \centering
  \includegraphics[width=\linewidth]{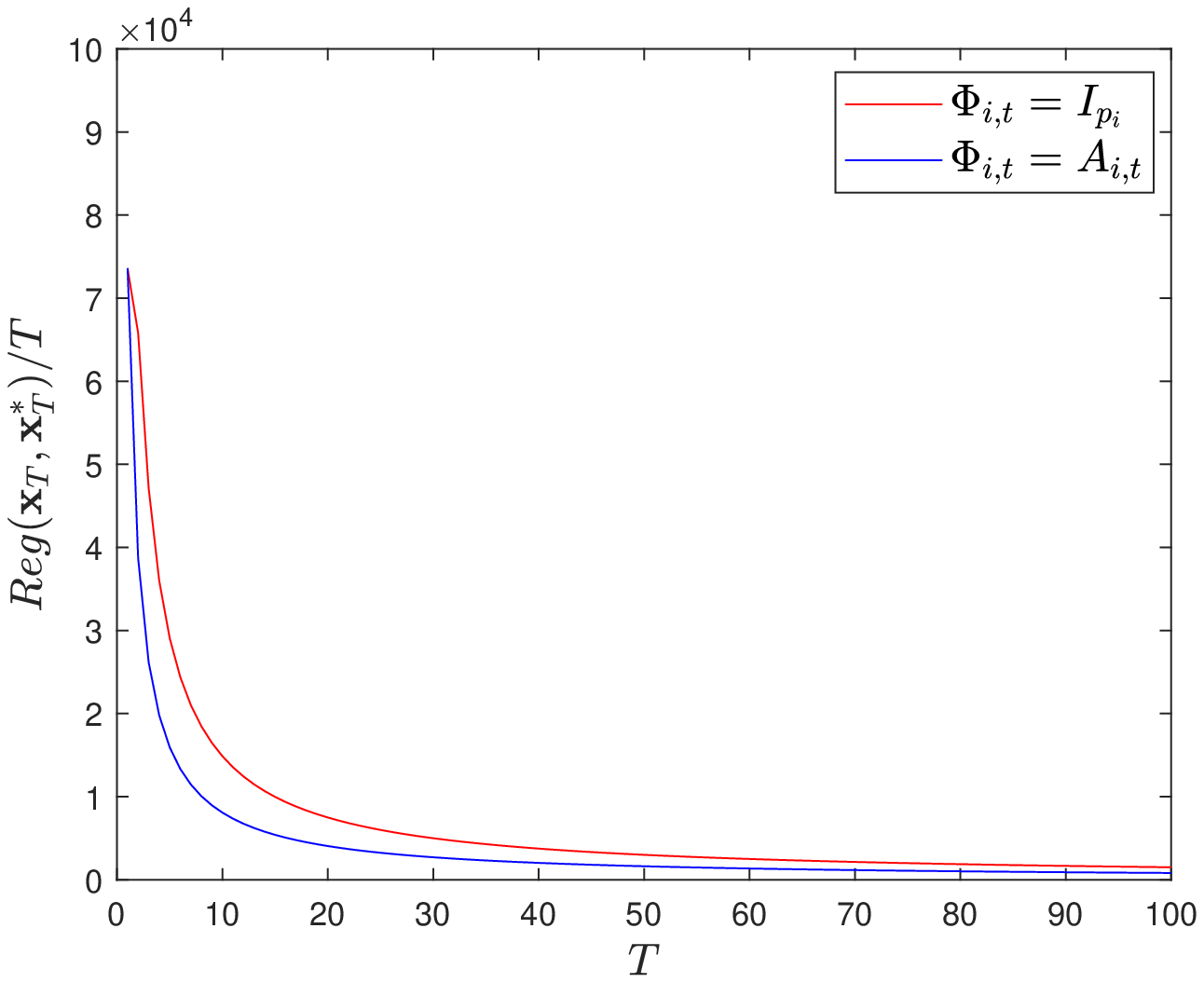}
  \caption{}
  \label{online:figPhireg}
\end{subfigure}%
\\
\begin{subfigure}{.5\textwidth}
  \centering
  \includegraphics[width=\linewidth]{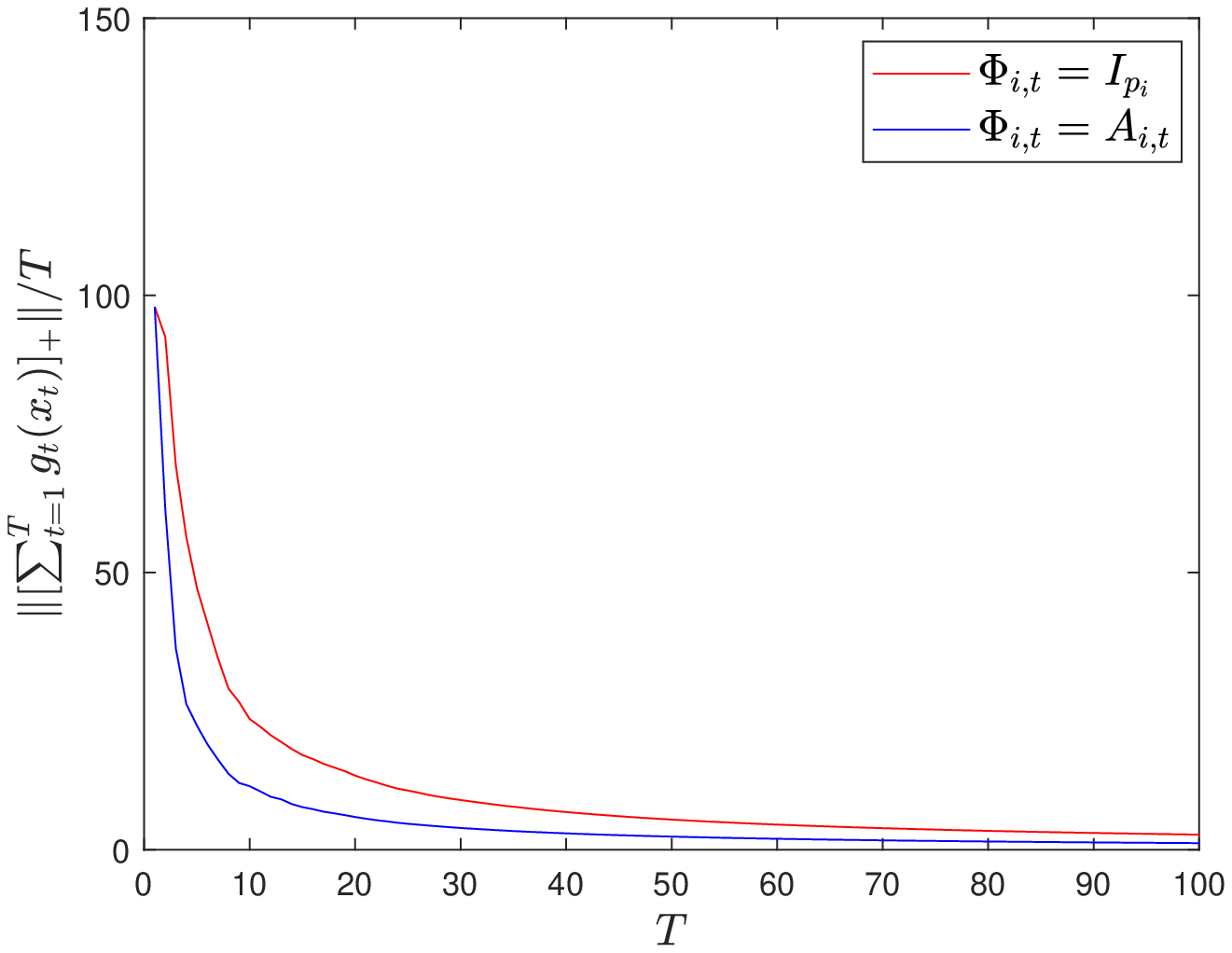}
  \caption{}
  \label{online:figPhiregc}
\end{subfigure}
\caption{Comparison of different $\Phi_{i,t}$: (a) Evolutions of $\Reg(\bsx_T,\bsx^*_T)/T$; (b) Evolutions of $\|[\sum_{t=1}^Tg_{t}(x_{t})]_+\|/T$.}
\label{online:figPhi}
\end{figure}

\subsection{Regularization Function}
To highlight the dependence of the dynamic regret and constraint violation with the regularization function, we run Algorithm~\ref{online_op:algorithm} for two cases. Case I: $f_{i,t}(x_i)=\zeta_{i,1}\langle \pi_{i,t},x_i\rangle+\zeta_{i,2}\|H_{i,t}x_{i}-y_{i,t}\|^2$, $r_{i,t}(x_i)=\lambda_{i,1}\|x_i\|_1+\lambda_{i,2}\|x_i\|^2$ and Case II: $f_{i,t}(x_i)=\zeta_{i,1}\langle \pi_{i,t},x_i\rangle+\zeta_{i,2}\|H_{i,t}x_{i}-y_{i,t}\|^2+\lambda_{i,1}\|x_i\|_1+\lambda_{i,2}\|x_i\|^2$, $r_{i,t}(x_i)=0$. Figs. \ref{online:figregula} (a) and (b) show the evolutions of $\Reg(\bsx_T,\bsx^*_T)/T$ and $\|[\sum_{t=1}^Tg_{t}(x_{t})]_+\|/T$, respectively, for these two cases. From these two figures, we can see that having the regularization term explicitly leads to smaller dynamic regret and constraint violation.

\begin{figure}
\begin{subfigure}{.5\textwidth}
  \centering
  \includegraphics[width=\linewidth]{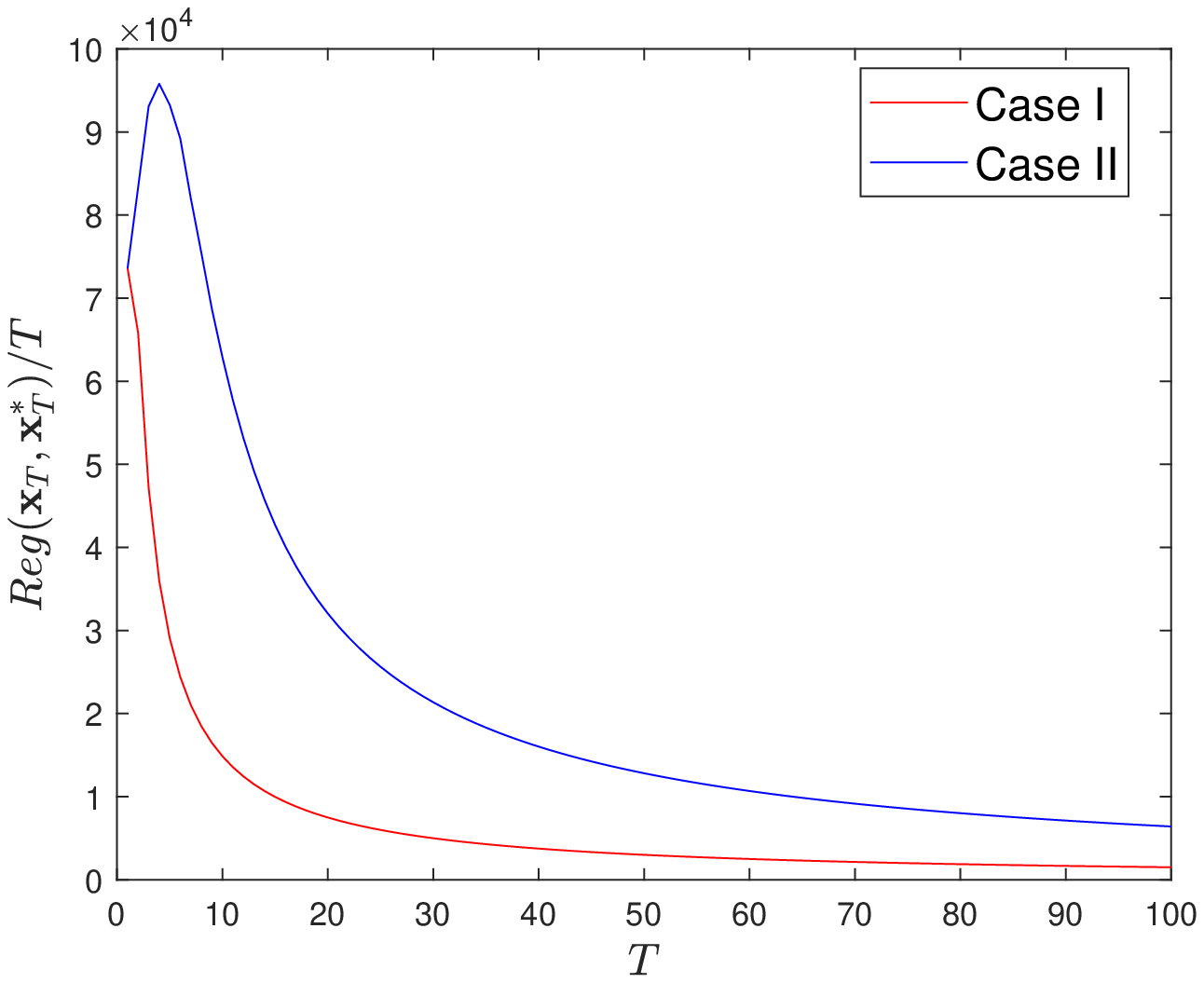}
  \caption{}
  \label{online:figregulareg}
\end{subfigure}%
\\
\begin{subfigure}{.5\textwidth}
  \centering
  \includegraphics[width=\linewidth]{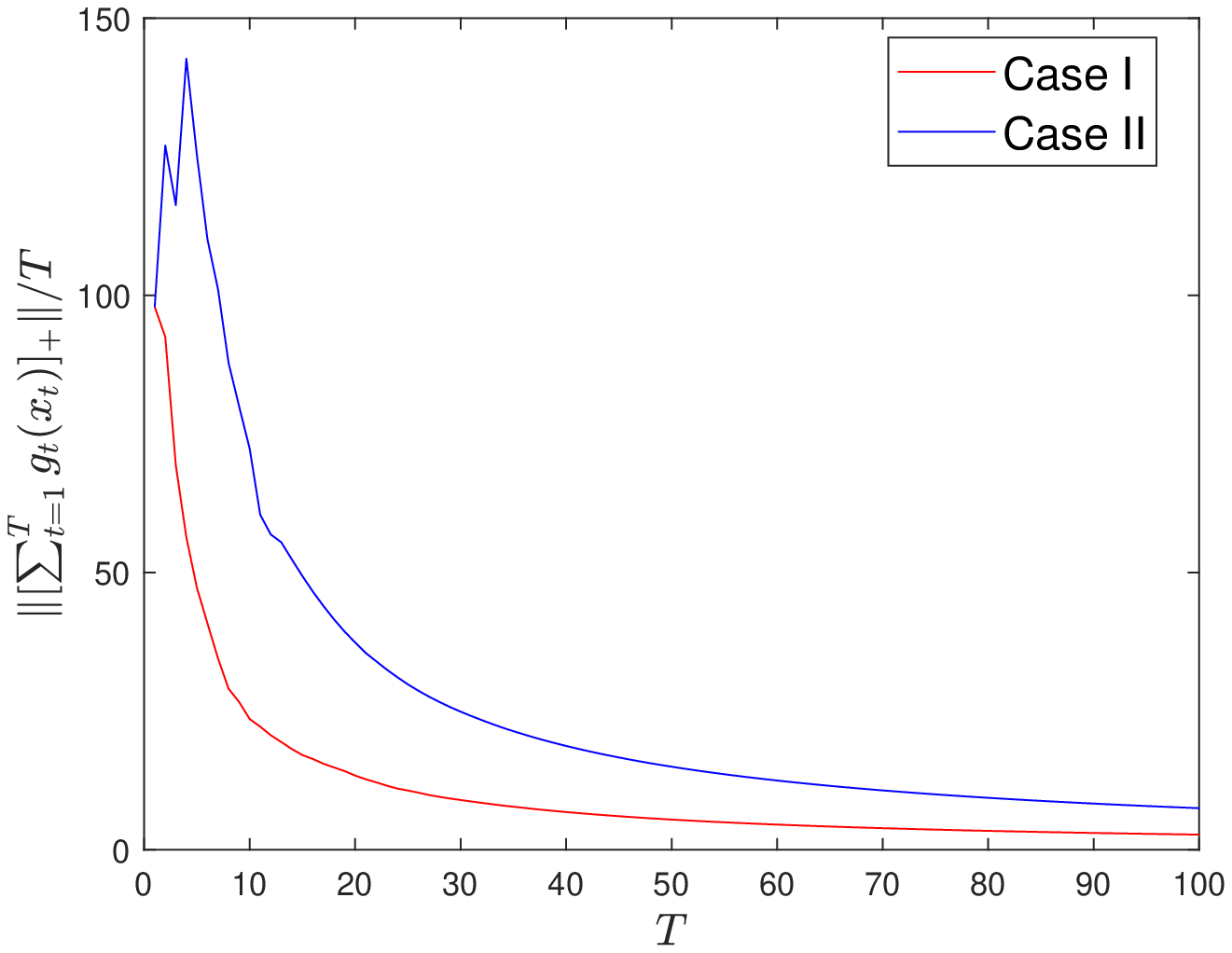}
  \caption{}
  \label{online:figregularegc}
\end{subfigure}
\caption{(a) Evolutions of $\Reg(\bsx_T,\bsx^*_T)/T$. (b) Evolutions of $\|[\sum_{t=1}^Tg_{t}(x_{t})]_+\|/T$.}
\label{online:figregula}
\end{figure}

\subsection{Effects of Parameter $\kappa$}
To investigate the dependence of the dynamic regret and constraint violation with the parameter $\kappa$, we run Algorithm~\ref{online_op:algorithm} with $\kappa=0.1,0.3,0.5,0.7,0.9$. Figs. \ref{online:figkappas} (a) and (b) show effects of $\kappa$ on  $\Reg(\bsx_T,\bsx^*_T)/T$ and $\|[\sum_{t=1}^Tg_{t}(x_{t})]_+\|/T$, respectively,  when $T=100,500,1000$. From these two figures, we can see that $\kappa$ almost does not affect $\Reg(\bsx_T,\bsx^*_T)/T$ and $\|[\sum_{t=1}^Tg_{t}(x_{t})]_+\|/T$ when $T$ is large (e.g., $T\ge500$). This phenomenon is not contradictory to the theoretical results shown in Theorem \ref{online_op:theoremstongconvex} since the theoretical results provide upper bounds of $\Reg(\bsx_T,\bsx^*_T)/T$ and $\|[\sum_{t=1}^Tg_{t}(x_{t})]_+\|/T$.

\begin{figure}
\begin{subfigure}{.5\textwidth}
  \centering
  \includegraphics[width=\linewidth]{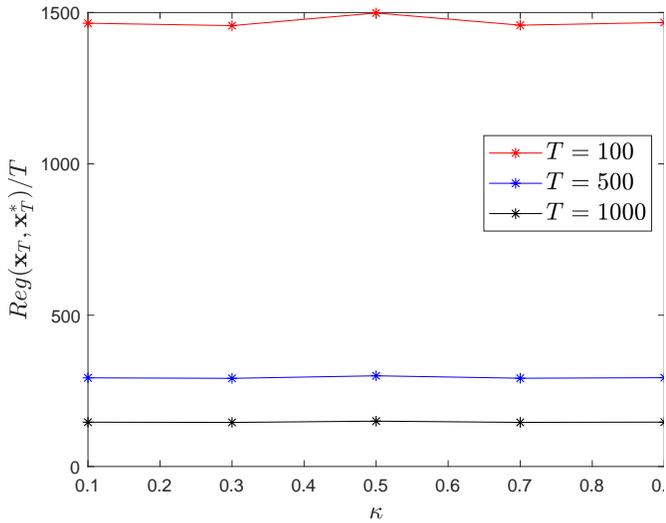}
  \caption{}
  \label{online:figkappasreg}
\end{subfigure}%
\\
\begin{subfigure}{.5\textwidth}
  \centering
  \includegraphics[width=\linewidth]{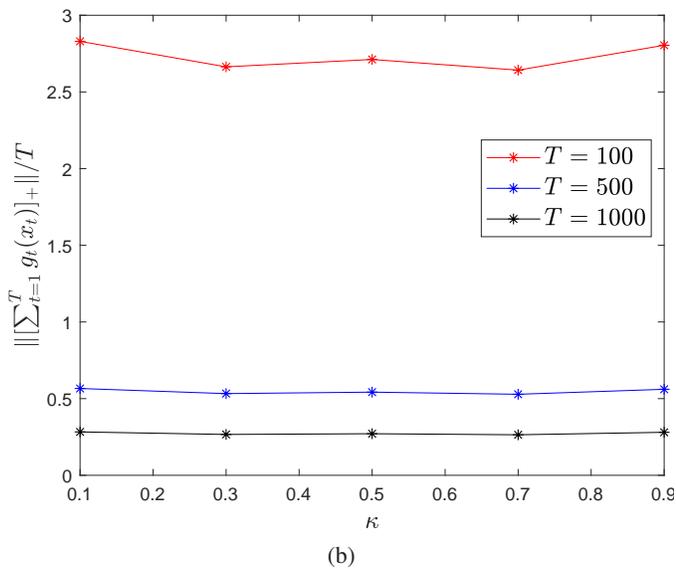}
  \caption{}
  \label{online:figkappasregc}
\end{subfigure}
\caption{Effects of parameter $\kappa$ on (a) $\Reg(\bsx_T,\bsx^*_T)/T$ and (b)  $\|[\sum_{t=1}^Tg_{t}(x_{t})]_+\|/T$ when $T=100,500,1000$.}
\label{online:figkappas}
\end{figure}

\subsection{Comparison to other Algorithms}
Since there are no distributed online algorithms to solve problem (\ref{online:problem1}), we compare Algorithm~\ref{online_op:algorithm} with the centralized online algorithms in \cite{sun2017safety,chen2017online,neely2017online}. Here, Algorithm~1 in \cite{sun2017safety} with $\alpha=10$, $\delta=1$, and $\mu=1/\sqrt{T}$, Algorithm~1 in \cite{chen2017online} with $\alpha=\mu=T^{-1/3}$, and the virtual queue algorithm in \cite{neely2017online} with $V=\sqrt{T}$ and $\alpha=V^2$ are used. Figs. \ref{online:figdisalgorithms} (a) and (b) show the evolutions of $\Reg(\bsx_T,\bsx^*_T)/T$ and $\|[\sum_{t=1}^Tg_{t}(x_{t})]_+\|/T$, respectively, for these algorithms. From these two figures, we can see that in this example Algorithm~\ref{online_op:algorithm} achieves smaller dynamic regret and constraint violation than the algorithms in \cite{chen2017online,neely2017online} and almost the same values as the algorithm in \cite{sun2017safety}.

\begin{figure}
\begin{subfigure}{.5\textwidth}
  \centering
  \includegraphics[width=\linewidth]{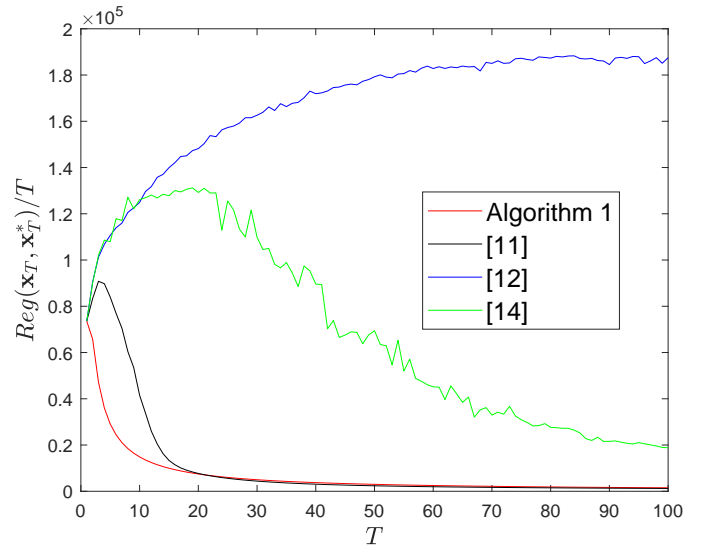}
  \caption{}
  \label{online:figdisalgorithmsreg}
\end{subfigure}%
\\
\begin{subfigure}{.5\textwidth}
  \centering
  \includegraphics[width=\linewidth]{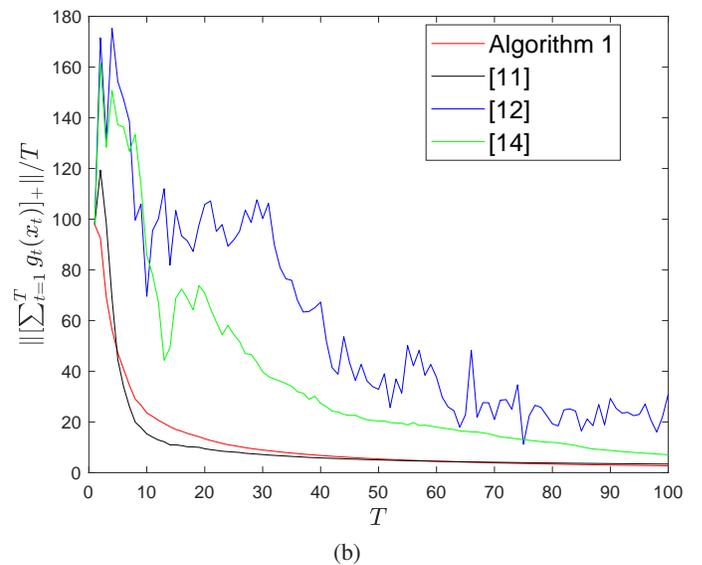}
  \caption{}
  \label{online:figdisalgorithmsregc}
\end{subfigure}
\caption{Comparison of other algorithms: (a) Evolutions of $\Reg(\bsx_T,\bsx^*_T)/T$; (b) Evolutions of $\|[\sum_{t=1}^Tg_{t}(x_{t})]_+\|/T$.}
\label{online:figdisalgorithms}
\end{figure}

\section{CONCLUSION}\label{online_opsec:conclusion}
In this paper, we considered an online convex optimization problem with time-varying coupled inequality constraints. We proposed a distributed online primal-dual dynamic mirror descent algorithm to solve this problem. We derived regret and constraint violation bounds for the algorithm  and showed how they depend on the stepsize sequences, the accumulated dynamic variation of the comparator sequence, the number of agents, and the network connectivity. We proved that the algorithm achieves sublinear regret and constraint violation for both arbitrary and strongly convex objective functions. We showed that the algorithm and results in this paper can be cast as extensions of existing algorithms.  Future research directions include extending the algorithm with bandit feedback and learning the dynamics of the optimal sequence.

\section*{ACKNOWLEDGMENTS}
The first author would like to thank the hospitality from the School of Electrical and Electronic Engineering, Nanyang Technological University during his visit March--June 2018.
The first author is also thankful to Dr. Tao Yang for discussions on distributed optimization.

\bibliographystyle{IEEEtran}
\bibliography{refs}

\appendix\label{online_op:appendix}

\subsection{Proof of Lemma~\ref{online_op:lemma_mirror}}\label{online_op:lemma_mirrorproof}
(i) Denote $\tilde{h}(x)=h(x)+\calD_\psi(x,z)$. Then $\tilde{h}$ is a convex function on $\Dom$.
Thus the optimality condition (\ref{online_op:lemma_mirroreq1}), i.e., $y=\argmin_{x\in\Dom}\tilde{h}(x)$, implies
$
\langle y-x,\nabla \tilde{h}(y)\rangle\le0,~\forall x\in\Dom
$.
Substituting $\nabla \tilde{h}(y)=\nabla h(y)+\nabla\psi(y)-\nabla\psi(z)$ into the above inequality yields
\begin{align*}
&\langle y-x,\nabla h(y)\rangle\le\langle y-x,\nabla\psi(z)-\nabla\psi(y)\rangle\\
=&\calD_\psi(x,z)-\calD_\psi(x,y)-\calD_\psi(y,z),~\forall x\in\Dom,
\end{align*}
where the equality holds since (\ref{online_op:bregmthree}). Hence, (\ref{online_op:lemma_mirroreq2}) holds.

(ii) $\tilde{h}(x)$ is strongly convex with convexity parameter $\sigma$ since $\calD_\psi$ is strongly convex. It is known that if $\tilde{h}:\Dom\rightarrow\mathbb{R}$ is a strongly convex function and is minimized at the point $x^{\min}\in\Dom$, then
\begin{align*}
\tilde{h}(x^{\min})\le \tilde{h}(x)-\frac{\sigma}{2}\|x-x^{\min}\|^2,~\forall x\in\Dom.
\end{align*}
Thus the optimality condition of (\ref{online_op:lemma_mirroreq1}) implies
\begin{align*}
h(y)+\calD_\psi(y,z)
\le h(z)+\calD_\psi(z,z)-\frac{\sigma}{2}\|z-y\|^2.
\end{align*}
Noting that $\calD_\psi(y,z)\ge\frac{\sigma}{2}\|z-y\|^2$ and $\calD_\psi(z,z)=0$, and rearranging the above inequality give
\begin{align}\label{online_op:cy}
\sigma\|z-y\|^2\le \frac{\sigma}{2}\|z-y\|^2+\calD_\psi(y,z)
\le h(z)-h(y).
\end{align}
From (\ref{online_op:subgradient}) and $\|\nabla h(x)\|\le G_h,~\forall x\in\Dom$, we have
\begin{align}\label{online_op:rcy}
h(z)-h(y)\le\langle\nabla h(z),z-y\rangle\le G_h\|z-y\|.
\end{align}
Thus, combining (\ref{online_op:cy}) and (\ref{online_op:rcy}) yields (\ref{online_op:lemma_mirrorine}).

\subsection{Proof of Lemma~\ref{online_op:lemma_virtualbound}}\label{online_op:lemma_virtualboundproof}
(i) We prove (\ref{online_op:lemma_virtualboundeqy}) by induction. %$\|\hat{Q}_{t,i}\|\le B^Q_t$, and $\|Q_{t,i}\|\le \frac{B^Q_t}{w}$ follows since $\hat{Q}_{t,i}=\sum_{j=1}^nW_{t,ij}Q_{t,j}$ and $W_{t,ii}>w$.

It is straightforward to see that $\tilde{q}_{i,1}=q_{i,1}={\bf 0}_{m},~\forall i\in[n]$, thus $\|\tilde{q}_{i,1}\|\le \frac{F}{\beta_1},~\|q_{i,1}\|\le \frac{F}{\beta_1},~\forall i\in[n]$. Assume that (\ref{online_op:lemma_virtualboundeqy}) is true at time $t$ for all $i\in[n]$. We show that it remains true at time $t+1$. The convexity of norms and $\sum_{j=1}^n[W_{t}]_{ij}=1$ yield
\begin{align*}
\|\tilde{q}_{i,t+1}\|\le&\sum_{j=1}^n[W_{t}]_{ij}\|q_{j,t}\|\le\sum_{j=1}^n[W_{t}]_{ij}\frac{ F}{\beta_{t}}\\
\le& \frac{ F}{\beta_{t+1}},~\forall i\in[n],
\end{align*}
where the last inequality holds due to the sequence $\{\beta_t\}$ is non-increasing.
(\ref{online_op:subgradient}) and (\ref{online_op:al_b}) imply
\begin{align}
&(1-\gamma_{t+1}\beta_{t+1})\tilde{q}_{i,t+1}+\gamma_{t+1}b_{i,t+1}\nonumber\\
&\le(1-\gamma_{t+1}\beta_{t+1})\tilde{q}_{i,t+1}+\gamma_{t+1}g_{i,t}(\tilde{x}_{i,t+1}).\label{online_op:qg}
\end{align}
Since $\|[x]_+\|\le\|y\|$ for all $x\le y$, (\ref{online_op:al_q}), (\ref{online_op:qg}), and (\ref{online_op:ftgtupper}) imply
\begin{align*}
&\|q_{i,t+1}\|\le(1-\gamma_{t+1}\beta_{t+1})\|\tilde{q}_{i,t+1}\|+\gamma_{t+1}\|g_{i,t}(\tilde{x}_{i,t+1})\|\\
&\le(1-\gamma_{t+1}\beta_{t+1})\frac{ F}{\beta_{t+1}}+\gamma_{t+1} F= \frac{ F}{\beta_{t+1}},~\forall i\in[n].
\end{align*}
Thus, (\ref{online_op:lemma_virtualboundeqy}) follows.

(ii) We can rewrite
(\ref{online_op:al_q}) as
\begin{align*}
q_{i,t+1}=\sum_{j=1}^n[W_{t}]_{ij}q_{j,t}+\epsilon^q_{i,t},
\end{align*}
where $\epsilon^q_{i,t}=[(1-\gamma_{t+1}\beta_{t+1})\tilde{q}_{i,t+1}+\gamma_{t+1}b_{i,t+1}]_{+}-\tilde{q}_{i,t+1}$.
From (\ref{online_op:ftgtupper}), (\ref{online_op:subgupper}), and (\ref{online_op:domainupper}), we have
\begin{align}\label{online_op:lemma_qbarequb}
\|b_{i,t+1}\|
&\le\|[g_{i,t}(x_{i,t})]_+\|+\|\nabla g_{i,t}(x_{i,t})\|
\|(\tilde{x}_{i,t+1}-x_{i,t})\|\nonumber\\
&\le F+Gd(X),~\forall i\in[n].
\end{align}
Thus, (\ref{online_op:proj}), (\ref{online_op:lemma_virtualboundeqy}), and (\ref{online_op:lemma_qbarequb}) give
\begin{align}\label{onlune_op:lemma_qbarequeps}
\|\epsilon^q_{i,t}\|
\le&\|-\gamma_{t+1}\beta_{t+1}\tilde{q}_{i,t+1}+\gamma_{t+1}b_{i,t+1}\|\nonumber\\
\le&B_1\gamma_{t+1},~\forall i\in[n].
\end{align}
Then, Lemma 2 in \cite{lee2017sublinear}, $q_{i,1}={\bf 0}_{m},~\forall i\in[n]$, and (\ref{onlune_op:lemma_qbarequeps}) yield
\begin{align*}
\|q_{i,t+1}-\bar{q}_{t+1}\|
\le n\tau B_1\sum_{s=1}^{t}\gamma_{s+1}\lambda^{t-s},~\forall i\in[n].
\end{align*}
So (\ref{online_op:lemma_qbarequ}) follows since $\sum_{j=1}^n[W_{t}]_{ij}=1$ and $\| \tilde{q}_{i,t+1}-\bar{q}_{t}\|=\|\sum_{j=1}^n[W_{t}]_{ij}q_{j,t}-\bar{q}_{t}\|\le \sum_{j=1}^n[W_{t}]_{ij}\|q_{j,t}-\bar{q}_{t}\|$.

(iii) Applying (\ref{online_op:proj}) to (\ref{online_op:al_q}) gives
\begin{align}
&\|q_{i,t}-q\|^2
\le\Big\|(1-\beta_t\gamma_t)\tilde{q}_{i,t}+\gamma_tb_{i,t}-q\Big\|^2\nonumber\\
&=\|\tilde{q}_{i,t}-q\|^2+(\gamma_t)^2\|b_{i,t}-\beta_t\tilde{q}_{i,t}\|^2\nonumber\\
&~~~+2\gamma_t[\tilde{q}_{i,t}]^\top\nabla g_{i,t-1}(x_{i,t-1})(\tilde{x}_{i,t}-x_{i,t-1})\nonumber\\
&~~~-2\gamma_tq^\top\nabla g_{i,t-1}(x_{i,t-1})(\tilde{x}_{i,t}-x_{i,t-1})\nonumber\\
&~~~+2\gamma_t[\tilde{q}_{i,t}-q]^\top g_{i,t-1}(x_{i,t-1})\nonumber\\
&~~~-2\beta_t\gamma_t[\tilde{q}_{i,t}-q]^\top\tilde{q}_{i,t}.\label{online_op:qmu}
\end{align}
For the first term on the right-hand side of the equality of (\ref{online_op:qmu}), by convexity of norms and $\sum_{j=1}^n[W_{t-1}]_{ij}=1$, it can be concluded that
\begin{align}
\|\tilde{q}_{i,t}-q\|^2=&\|\sum_{j=1}^n[W_{t-1}]_{ij}q_{j,t-1}-\sum_{j=1}^n[W_{t-1}]_{ij}q\|^2\nonumber\\
\le&\sum_{j=1}^n[W_{t-1}]_{ij}\|q_{j,t-1}-q\|^2.\label{online_op:qmu0}
\end{align}
For the second term, (\ref{online_op:lemma_virtualboundeqy}) and (\ref{online_op:lemma_qbarequb}) yield
\begin{align}
(\gamma_t)^2\|b_{i,t}-\beta_t\tilde{q}_{i,t}\|^2
\le(B_1\gamma_t)^2.\label{online_op:qmu1}
\end{align}
For the fourth term,  (\ref{online_op:subgupper}), and the Cauchy-Schwarz  inequality yield
\begin{align}
&-2\gamma_tq^\top\nabla g_{i,t-1}(x_{i,t-1})(\tilde{x}_{i,t}-x_{i,t-1})\nonumber\\
&\le2\gamma_t\Big(\frac{G^2\alpha_t}{\underline{\sigma}}\|q\|^2
+\frac{\underline{\sigma}}{4\alpha_t}\|\tilde{x}_{i,t}-x_{i,t-1}\|^2\Big).\label{online_op:qmu2}
\end{align}
For the fifth term, we have
\begin{align}
&2\gamma_t[\tilde{q}_{i,t}-q]^\top g_{i,t-1}(x_{i,t-1})
=2\gamma_t[\bar{q}_{t-1}-q]^\top g_{i,t-1}(x_{i,t-1})\nonumber\\
&+2\gamma_t[\tilde{q}_{i,t}-\bar{q}_{t-1}]^\top g_{i,t-1}(x_{i,t-1}).\label{online_op:qmu4}
\end{align}
Moreover, from (\ref{online_op:ftgtupper}) and (\ref{online_op:lemma_qbarequ}), we have
\begin{align}
&2\gamma_t[\tilde{q}_{i,t}-\bar{q}_{t-1}]^\top g_{i,t-1}(x_{i,t-1})\nonumber\\
&\le2\gamma_t\|\tilde{q}_{i,t}-\bar{q}_{t-1}\|\|g_{i,t-1}(x_{i,t-1})\|
\le\frac{2\gamma_tE_{1}(t)}{n}.\label{online_op:qmu5}
\end{align}
For the last term in the equality of (\ref{online_op:qmu}),  neglecting the nonnegative term $\beta_t\gamma_t\|\tilde{q}_{i,t}\|^2$ gives
\begin{align}
-2\beta_t\gamma_t[\tilde{q}_{i,t}-q]^\top\tilde{q}_{i,t}
\le\beta_t\gamma_t(\|q\|^2-\|\tilde{q}_{i,t}-q\|^2).\label{online_op:qmu3}
\end{align}
Then, combining (\ref{online_op:qmu})--(\ref{online_op:qmu3}), summing over $i\in[n]$, and dividing by $2\gamma_t$, and using $\sum_{i=1}^n[W_{t-1}]_{ij}=1,~\forall t\in\mathbb{N}_+$, yield (\ref{online_op:gvirtualnorm}).

\subsection{Proof of Lemma~\ref{online_op:lemma_regretdelta}}\label{online_op:lemma_regretdeltaproof}
From (\ref{online_op:subgradient}), we have
\begin{align}\label{online_op:fxy}
&f_{i,t}(x_{i,t})+r_{i,t}(x_{i,t})-f_{i,t}(y_{i,t})-r_{i,t}(y_{i,t})\nonumber\\
&=f_{i,t}(x_{i,t})-f_{i,t}(y_{i,t})+r_{i,t}(x_{i,t})-r_{i,t}(\tilde{x}_{i,t+1})\nonumber\\
&~~~+r_{i,t}(\tilde{x}_{i,t+1})-r_{i,t}(y_{i,t})\nonumber\\
&\le\langle\nabla f_{i,t}(x_{i,t}),x_{i,t}-y_{i,t}\rangle
+\langle\nabla r_{i,t}(x_{i,t}),x_{i,t}-\tilde{x}_{i,t+1}\rangle\nonumber\\
&~~~+\langle\nabla r_{i,t}(\tilde{x}_{i,t+1}),\tilde{x}_{i,t+1}-y_{i,t}\rangle\nonumber\\
&=\langle\nabla f_{i,t}(x_{i,t})+\nabla r_{i,t}(x_{i,t}),x_{i,t}-\tilde{x}_{i,t+1}\rangle\nonumber\\
&~~~+\langle\nabla f_{i,t}(x_{i,t})+\nabla r_{i,t}(\tilde{x}_{i,t+1}),\tilde{x}_{i,t+1}-y_{i,t}\rangle.
\end{align}
We now bound each of the two terms above. For the first term, (\ref{online_op:subgupper}) and the Cauchy-Schwarz  inequality give
\begin{align}
&\langle\nabla f_{i,t}(x_{i,t})+\nabla r_{i,t}(x_{i,t}),x_{i,t}-\tilde{x}_{i,t+1}\rangle\nonumber\\
&\le2G\|x_{i,t}-\tilde{x}_{i,t+1}\|\nonumber\\
&\le\frac{\underline{\sigma}}{4\alpha_{t+1}}\|x_{i,t}-\tilde{x}_{i,t+1}\|^2
+\frac{4G^2\alpha_{t+1}}{\underline{\sigma}}.\label{online_op:fxx}
\end{align}
For the second term, we have
\begin{align}
&\langle\nabla f_{i,t}(x_{i,t})+\nabla r_{i,t}(\tilde{x}_{i,t+1}),\tilde{x}_{i,t+1}-y_{i,t}\rangle\nonumber\\
&=\langle(\nabla g_{i,t}(x_{i,t}))^\top \tilde{q}_{i,t+1},y_{i,t}-\tilde{x}_{i,t+1}\rangle\nonumber\\
&~~~+\langle a_{i,t+1}+\nabla r_{i,t}(\tilde{x}_{i,t+1}),\tilde{x}_{i,t+1}-y_{i,t}\rangle\nonumber\\
&=\langle(\nabla g_{i,t}(x_{i,t}))^\top  \tilde{q}_{i,t+1},y_{i,t}-x_{i,t}\rangle\nonumber\\
&~~~+\langle(\nabla g_{i,t}(x_{i,t}))^\top  \tilde{q}_{i,t+1},x_{i,t}-\tilde{x}_{i,t+1}\rangle\nonumber\\
&~~~+\langle a_{i,t+1}+\nabla r_{i,t}(\tilde{x}_{i,t+1}),\tilde{x}_{i,t+1}-y_{i,t}\rangle.\label{online_op:fxy1}
\end{align}
From (\ref{online_op:subgradient}) and $\tilde{q}_{i,t}\ge{\bf0}_{m},~\forall t\in\mathbb{N}_+,~\forall i\in[n]$, we have
\begin{align}
&\langle(\nabla g_{i,t}(x_{i,t}))^\top  \tilde{q}_{i,t+1},y_{i,t}-x_{i,t}\rangle\nonumber\\
&\le[ \tilde{q}_{i,t+1}]^\top g_{i,t}(y_{{i,t}})-[ \tilde{q}_{i,t+1}]^\top g_{i,t}(x_{i,t})\nonumber\\
&=[\bar{q}_{t}]^\top [g_{i,t}(y_{i,t})- g_{i,t}(x_{i,t})]\nonumber\\
&~~~+[ \tilde{q}_{i,t+1}-\bar{q}_{t}]^\top [g_{i,t}(y_{i,t})- g_{i,t}(x_{i,t})].\label{online_op:gyx}
\end{align}
Similar to (\ref{online_op:qmu5}), we have
\begin{align}\label{online_op:qqbar}
[\tilde{q}_{i,t+1}-\bar{q}_{t}]^\top [g_{i,t}(y_{i,t})- g_{i,t}(x_{i,t})]
\le
\frac{2E_{1}(t+1)}{n}.
\end{align}
Applying (\ref{online_op:lemma_mirroreq2}) to the update rule (\ref{online_op:al_x}), we get
\begin{align}
&\langle a_{i,t+1}+\nabla r_{i,t}(\tilde{x}_{i,t+1}),\tilde{x}_{i,t+1}-y_{i,t}\rangle\nonumber\\
&\le\frac{1}{\alpha_{t+1}}[\calD_{\psi_i}(y_{i,t},x_{i,t})
-\calD_{\psi_i}(y_{i,t},\tilde{x}_{i,t+1})\nonumber\\
&~~~-\calD_{\psi_i}(\tilde{x}_{i,t+1},x_{i,t})]\nonumber\\
&=\frac{1}{\alpha_{t+1}}[\calD_{\psi_i}(y_{i,t},x_{i,t})
-\calD_{\psi_i}(y_{i,t+1},x_{i,t+1})\nonumber\\
&~~~+\calD_{\psi_i}(y_{i,t+1},x_{i,t+1})-\calD_{\psi_i}(\Phi_{i,t+1}(y_{i,t}),x_{i,t+1})\nonumber\\
&~~~+\calD_{\psi_i}(\Phi_{i,t+1}(y_{i,t}),x_{i,t+1})-\calD_{\psi_i}(y_{i,t},\tilde{x}_{i,t+1})\nonumber\\
&~~~-\calD_{\psi_i}(\tilde{x}_{i,t+1},x_{i,t})]\nonumber\\
&\le\frac{1}{\alpha_{t+1}}[\calD_{\psi_i}(y_{i,t},x_{i,t})
-\calD_{\psi_i}(y_{i,t+1},x_{i,t+1})\nonumber\\
&~~~+K\|y_{i,t+1}-\Phi_{i,t+1}(y_{i,t})\|
-\frac{\underline{\sigma}}{2}\|\tilde{x}_{i,t+1}-x_{i,t}\|^2],\label{online_op:omgea2}
\end{align}
where the last inequality holds since (\ref{online_op:al_xat}), (\ref{online_op:assnonexpansiveequ}), (\ref{online_op:bregmalip}), and (\ref{online_op:eqbergman}).

Combining (\ref{online_op:fxy})--(\ref{online_op:omgea2}) and summing over $i\in[n]$ yield (\ref{online_op:lemma_regretdeltaequ}).

\subsection{Proof of Lemma~\ref{online_op:theoremreg}}\label{online_op:theoremregproof}
(i) The definition of $\Delta_t$ gives
\begin{align}
-\frac{\Delta_t}{2\gamma_t}
=&\frac{1}{2\gamma_t}\sum_{i=1}^n[(1-\beta_t\gamma_t)\|q_{i,t-1}-q\|^2-\|q_{i,t}-q\|^2]\nonumber\\
=&\frac{1}{2}\sum_{i=1}^n\Big[\frac{1}{\gamma_{t-1}}\|q_{i,t-1}-q\|^2
-\frac{1}{\gamma_{t}}\|q_{i,t}-q\|^2\Big]\nonumber\\
&+\frac{1}{2}\sum_{i=1}^n\Big(\frac{1}{\gamma_{t}}
-\frac{1}{\gamma_{t-1}}-\beta_t\Big)\|q_{i,t-1}-q\|^2.\label{online_op:qmu7}
\end{align}

For any nonnegative sequence $\zeta_1,\zeta_2,\dots$, it holds that
\begin{align}\label{online_op:zeta}
\sum_{t=1}^T\sum_{s=1}^{t}\zeta_{s+1}\lambda^{t-s}
=\sum_{t=1}^{T}\zeta_{t+1}\sum_{s=0}^{T-t}\lambda^{s}
\le\frac{1}{(1-\lambda)}\sum_{t=1}^T\zeta_{t+1}.
\end{align}

Let $g_c:\mathbb{R}^m_{+}\rightarrow\mathbb{R}$ be a function defined as
\begin{align}\label{online_op:gc}
g_c(q)=&\Big[\sum_{t=1}^Tg_{t}(x_{t})\Big]^\top q\nonumber\\
&-n\Big[\frac{1}{\gamma_1}+
\sum_{t=1}^T\Big(\frac{G^2\alpha_{t+1}}{\underline{\sigma}}
+\frac{\beta_{t+1}}{2}\Big)\Big]\|q\|^2.
\end{align}
Combining (\ref{online_op:gvirtualnorm}) and (\ref{online_op:lemma_regretdeltaequ}), summing over $t\in[T]$, neglecting the nonnegative term $\|q_{i,T+1}-q\|^2$, and using (\ref{online_op:qmu7})--(\ref{online_op:gc}),  $\|q_{i,1}-q\|^2\le2\|q_{i,1}\|^2+2\|q\|^2=2\|q\|^2$,  and $g_{t}(y_t)\le{\bf 0}_{m},~\bsy_T\in\calX_{T}$, yield
\begin{align}\label{online_op:theoremconsequ2}
&g_c(q)+\Reg(\bsx_T,\bsy_T)\nonumber\\
&\le C_{1,1}\sum_{t=1}^T\gamma_{t+1}
+\frac{4nG^2}{\underline{\sigma}}\sum_{t=1}^T\alpha_{t+1}+\sum_{t=1}^TE_{2}(t+1)\nonumber\\
&~~~-\frac{1}{2}\sum_{t=1}^T\sum_{i=1}^n\Big(\frac{1}{\gamma_{t}}
-\frac{1}{\gamma_{t+1}}+\beta_{t+1}\Big)\|q_{i,t}-q\|^2\nonumber\\
&~~~+K\sum_{t=1}^T\sum_{i=1}^n
\frac{\|y_{i,t+1}-\Phi_{i,t+1}(y_{i,t})\|}{\alpha_{t+1}},~\forall q\in\mathbb{R}^m_{+}.
\end{align}
Then, substituting $q={\bf 0}_{m}$ into (\ref{online_op:theoremconsequ2}), setting $y_{i,T+1}=\Phi_{i,T+1}(y_{i,T})$, noting that $\{\alpha_t\}$ is non-increasing, and rearranging terms yield (\ref{online_op:theoremregequ}).

%\subsection{Proof of Proposition~\ref{online_op:theoremcons}}\label{online_op:theoremconsproof}
(ii) Substituting $q=q_0$ into $g_c(q)$ gives
\begin{align}\label{online_op:gcequ}
g_c(q_0)=&\frac{\|[\sum_{t=1}^Tg_{t}(x_{t})]_+\|^2}
{4n[\frac{1}{\gamma_1}
+\sum_{t=1}^T(\frac{G^2\alpha_{t+1}}{\underline{\sigma}}+\frac{\beta_{t+1}}{2})]}.
\end{align}
Moreover, (\ref{online_op:ftgtupper}) gives
\begin{align}
|\Reg(\bsx_T,\bsy_T)|\le&2nFT,~\forall \bsy_T\in\calX_T.\label{online_op:ff}
\end{align}
Substituting $q=q_0$ into (\ref{online_op:theoremconsequ2}), combining (\ref{online_op:gcequ})--(\ref{online_op:ff}), and rearranging terms give (\ref{online_op:theoremconsequ}).

\subsection{Proof of Theorem~\ref{online_op:corollaryreg}}\label{online_op:corollaryregproof}
(i) For any constant $\kappa<1$ and $T\in\mathbb{N}_+$, it holds that
\begin{align}\label{online_op:sequenceupp}
\sum_{t=1}^T\frac{1}{t^\kappa}\le\int_1^T\frac{1}{t^\kappa}dt+1
=\frac{T^{1-\kappa}-\kappa}{1-\kappa}\le\frac{T^{1-\kappa}}{1-\kappa}.
\end{align}
Applying (\ref{online_op:sequenceupp}) to the first three terms in the right-hand side of (\ref{online_op:theoremregequ}) gives
\begin{align}
C_{1,1}\sum_{t=1}^T\gamma_{t+1}\le&\frac{C_{1,1}}{\kappa}T^{\kappa},\label{online_op:corollaryregequ11}\\
C_{1,2}\sum_{t=1}^T\alpha_{t+1}\le&\frac{C_{1,2}}{1-c}T^{1-c}.\label{online_op:corollaryregequ12}
\end{align}
Noting that $\{\alpha_t\}$ is non-increasing and (\ref{online_op:bregmanupp}), for any $s\in[T]$, we have
\begin{align}
&\sum_{t=s}^TE_{2}(t+1)\nonumber\\
&=\sum_{t=s}^T\sum_{i=1}^n\Big[\frac{1}{\alpha_{t}}\calD_{\psi_i}(y_{i,t},x_{i,t})
-\frac{1}{\alpha_{t+1}}\calD_{\psi_i}(y_{i,t+1},x_{i,t+1})\Big]\nonumber\\
&~~~+\sum_{t=s}^T\sum_{i=1}^n\Big(\frac{1}{\alpha_{t+1}}-\frac{1}{\alpha_{t}}\Big)
\calD_{\psi_i}(y_{i,t},x_{i,t})\nonumber\\
&\le\frac{1}{\alpha_{s}}\sum_{i=1}^n\calD_{\psi_i}(y_{i,s},x_{i,s})
-\frac{1}{\alpha_{T+1}}\sum_{i=1}^n\calD_{\psi_i}(y_{i,T+1},x_{i,T+1})\nonumber\\
&~~~+n\Big(\frac{1}{\alpha_{T+1}}
-\frac{1}{\alpha_{s}}\Big)d(X)K\le\frac{nd(X)K}{\alpha_{T+1}}.\label{online_op:dyz}
\end{align}
Combining (\ref{online_op:theoremregequ}) and (\ref{online_op:corollaryregequ11})--(\ref{online_op:dyz}), setting $y_{i,t}=x^*_{i,t},~\forall t\in[T]$, and noting that the second last term in the right-hand side of (\ref{online_op:theoremregequ}) is non-positive since $\frac{1}{\gamma_{t}}-\frac{1}{\gamma_{t+1}}
+\beta_{t+1}>0$ yield (\ref{online_op:corollaryregequ1}).%the result as it is stated.

(ii) Using (\ref{online_op:sequenceupp}) gives
\begin{align}
&4n\Big[\frac{1}{\gamma_1}
+\sum_{t=1}^T\Big(\frac{G^2\alpha_{t+1}}{\underline{\sigma}}+\frac{\beta_{t+1}}{2}\Big)\Big]
\le C_{2,1}T^{\max\{1-c,1-\kappa\}}.\label{online_op:corollaryconsequ1}
\end{align}

Combining (\ref{online_op:theoremconsequ}) and (\ref{online_op:corollaryregequ11})--(\ref{online_op:corollaryconsequ1})
and noting that the last term in the right-hand side of (\ref{online_op:theoremconsequ}) is non-positive since $\frac{1}{\gamma_{t}}-\frac{1}{\gamma_{t+1}}
+\beta_{t+1}>0$ give (\ref{online_op:corollaryconsequ}).

\subsection{Proof of Theorem~\ref{online_op:theoremslater}}\label{online_op:theoremslaterproof}

(i) Substituting $c=1-\kappa$ in (\ref{online_op:corollaryregequ1}) gives (\ref{online_op:regslater}).

(ii) We first show that $\|q_t\|\le B_2$ by induction, where $q_t=\col(q_{1,t},\dots,q_{n,t})$.

It is straightforward to see that $\|q_1\|=0\le B_2$. Suppose that there exists $T_1\in\mathbb{N}_+$ such that $\|q_t\|\le B_2,~\forall t\in[T_1]$. We show that $\|q_{T_1+1}\|\le B_2$ by contradiction.
Now suppose that $\|q_{T_1+1}\|> B_2$.
Noting that $\|\bar{q}_{T_1+1}\|_1=\|q_{T_1+1}\|_1\ge\|q_{T_1+1}\|>B_2$ and $\|\bar{q}_1\|_1=0$, we know that there exists $t_0\in[T_1]$ such that $\|\bar{q}_{t_0}\|_1\le\frac{B_2}{2}$. Let $t_1=\max\{t_0:~\|\bar{q}_{t_0}\|_1\le\frac{B_2}{2},~t_0\in[T_1]\}$.
Combining (\ref{online_op:gvirtualnorm}) and (\ref{online_op:lemma_regretdeltaequ}), substituting $q={\bf 0}_{m}$ and $y_t=x_0$, setting $\{\Phi_{t,i}\}$ as the identity mapping, and using $|f_{t}(x_{t})-f_{t}(x_{0})|\le2F$, and (\ref{online_op:gtcon}) yield
\begin{align}\label{online_op:theoremslaterqbar3}
&\|q_{t+1}\|^2-(1-\beta_{t+1}\gamma_{t+1})\|q_t\|^2\nonumber\\
&\le 2B_3\gamma_{t+1}+2\gamma_{t+1}E_2(t+1)-2\varepsilon\|\bar{q}_t\|_1\gamma_{t+1}.
\end{align}
Summing (\ref{online_op:theoremslaterqbar3}) over $t\in\{t_1,\dots,T_1\}$, using (\ref{online_op:bregmanupp}), $\alpha_t=\gamma_t=\frac{1}{t^{1-\kappa}}$ and $\beta_t\ge0$, and noting that $\|q_{T_1+1}\|> B_2$, $\|q_{t_1}\|\le\|\bar{q}_{t_1}\|_1\le \frac{B_2}{2}$, and $\|\bar{q}_t\|_1>\frac{B_2}{2},~\forall t\in\{t_1+1,\dots,T_1\}$ give
\begin{align}
&\frac{3(B_2)^2}{4}<\|q_{T_1+1}\|^2
-\|q_{t_1}\|^2+\sum_{t=t_1}^{T_1}\beta_{t+1}\gamma_{t+1}\|q_{t}\|^2\nonumber\\
&\le 2B_3\sum_{t=t_1}^{T_1}\gamma_{t+1}+2nd(X)K
-2\varepsilon\sum_{t=t_1}^{T_1}\|\bar{q}_t\|_1\gamma_{t+1}\nonumber\\
&\le \frac{2B_3}{\kappa}[(T_1+1)^{\kappa}-(t_1+1)^{\kappa}]+2B_3+2nd(X)K\nonumber\\
&~~~-\frac{\varepsilon B_2}{\kappa}[(T_1+1)^{\kappa}-(t_1+1)^{\kappa}]+\varepsilon B_2-2\varepsilon\|\bar{q}_{t_1}\|_1\nonumber\\
&\le2nd(X)K+2\varepsilon B_2\le\frac{(B_2)^2}{2},
\end{align}
which is a contradiction. Thus, $\|q_{T_1+1}\|\le B_2$.

We now show (\ref{online_op:regcslater}) holds.
Applying (\ref{online_op:lemma_mirrorine}) to the update rule (\ref{online_op:al_x}) and noting $\|\tilde{q}_{i,t+1}\|\le\|q_t\|\le B_2$ give
\begin{align}
\|\tilde{x}_{i,t+1}-x_{i,t}\|\le&\frac{\|\alpha_{t+1}a_{i,t+1}\|+\alpha_{t+1}G}{\underline{\sigma}}\nonumber\\
\le& \frac{G\alpha_{t+1}}{\underline{\sigma} }\Big(B_2+2\Big).\label{online_op:xxplusoneslater}
\end{align}
(\ref{online_op:al_qhat}) and (\ref{online_op:al_q}) give
\begin{align}
q_{i,t+1}
\ge(1-\beta_{t+1}\gamma_{t+1})\sum_{j=1}^n[W_{t}]_{ij}q_{j,t}+\gamma_{t+1}b_{i,t+1}.\label{online_op:theoremslaterq}
\end{align}
Summing (\ref{online_op:theoremslaterq}) over $i\in[n]$, dividing by $n\gamma_{t+1}$, and using $\sum_{i=1}^n[W_{t}]_{ij}=1,~\forall t\in\mathbb{N}_+$, (\ref{online_op:subgupper}), (\ref{online_op:al_b}), and (\ref{online_op:xxplusoneslater}) yield
\begin{align}
\frac{\bar{q}_{t+1}}{\gamma_{t+1}}\ge&(\frac{1}{\gamma_{t+1}}-\beta_{t+1})\bar{q}_{t}
+\frac{1}{n}\sum_{i=1}^nb_{i,t+1}\nonumber\\
\ge&(\frac{1}{\gamma_{t+1}}-\beta_{t+1})\bar{q}_{t}+\frac{1}{n}g_{t}(x_{t})\nonumber\\
&-\frac{G^2\alpha_{t+1}}{\underline{\sigma} }\Big(B_2+2\Big){\bf 1}_m.\label{online_op:theoremslaterqbar}
\end{align}
Summing (\ref{online_op:theoremslaterqbar}) over $t\in[T]$ gives
\begin{align}
\frac{1}{n}\sum_{t=1}^Tg_{t}(x_{t})\le&\frac{\bar{q}_{T+1}}{\gamma_{T+1}}
+\sum_{t=1}^T\beta_{t+1}\bar{q}_{t}\nonumber\\
&+\sum_{t=1}^T\frac{G^2\alpha_{t+1}}{\underline{\sigma} }\Big(B_2+2\Big){\bf 1}_m.
\end{align}
Noting that $\|[x]_+\|\le\|y\|$ for all $x\le y$ and using $\|\bar{q}_t\|\le\|q_t\|\le B_2$ and (\ref{online_op:sequenceupp}) yield (\ref{online_op:regcslater}).

\subsection{Proof of Theorem~\ref{online_op:theoremstongconvex}}\label{online_op:theoremstongconvexproof}

(i) We first show that $\Reg(\bsx_T,\check{\bsx}_T^*)\le C_4T^{\kappa}$ when $\alpha_t=\frac{1}{t^{1-\kappa}}$.

Under Assumption~\ref{online_op:assstrongconvex}, (\ref{online_op:fxy}) can be replaced by
\begin{align}\label{online_op:fxystrongconvex}
&f_{i,t}(x_{i,t})+r_{i,t}(x_{i,t})-f_{i,t}(y_{i,t})-r_{i,t}(y_{i,t})\nonumber\\
&\le\langle\nabla f_{i,t}(x_{i,t}),x_{i,t}-y_{i,t}\rangle
+\langle\nabla r_{i,t}(x_{i,t}),x_{i,t}-\tilde{x}_{i,t+1}\rangle\nonumber\\
&~~~+\langle\nabla r_{i,t}(\tilde{x}_{i,t+1}),\tilde{x}_{i,t+1}-y_{i,t}\rangle
-\underline{\mu}\calD_{\psi_i}(y_{i,t},x_{i,t})\nonumber\\
&=\langle\nabla f_{i,t}(x_{i,t})+\nabla r_{i,t}(x_{i,t}),x_{i,t}-\tilde{x}_{i,t+1}\rangle\nonumber\\
&~~~+\langle\nabla f_{i,t}(x_{i,t})+\nabla r_{i,t}(\tilde{x}_{i,t+1}),\tilde{x}_{i,t+1}-y_{i,t}\rangle\nonumber\\
&~~~-\underline{\mu}\calD_{\psi_i}(y_{i,t},x_{i,t}).
\end{align}
Thus, (\ref{online_op:lemma_regretdeltaequ})--(\ref{online_op:theoremconsequ}) still hold if replacing $E_{2}(t+1)$ by
\begin{align*}
E_{3}(t+1)=&\sum_{i=1}^n\Big\{\frac{1}{\alpha_{t+1}}\big[\calD_{\psi_i}(y_{i,t},x_{i,t})\\
&-\calD_{\psi_i}(y_{i,t+1},x_{i,t+1})\big]-\underline{\mu}\calD_{\psi_i}(y_{i,t},x_{i,t})\Big\}.
\end{align*}
Then,
\begin{align}\label{online_op:e3}
&\sum_{t=1}^TE_{3}(t+1)\nonumber\\
&=\sum_{t=1}^T\sum_{i=1}^n\Big[\frac{1}{\alpha_{t}}\calD_{\psi_i}(y_{i,t},x_{i,t})
-\frac{1}{\alpha_{t+1}}\calD_{\psi_i}(y_{i,t+1},x_{i,t+1})\Big]\nonumber\\
&~~~+\sum_{t=1}^T\sum_{i=1}^n\Big(\frac{1}{\alpha_{t+1}}
-\frac{1}{\alpha_{t}}-\underline{\mu}\Big)\calD_{\psi_i}(y_{i,t},x_{i,t}).
\end{align}
Noting that $\underline{\mu}>0$, $\calD_{\psi_i}(\cdot,\cdot)\ge0$, and $\frac{1}{\alpha_{t+1}}
-\frac{1}{\alpha_{t}}-\underline{\mu}=\frac{t+1}{(t+1)^\kappa}-\frac{t}{t^\kappa}-\underline{\mu}
<\frac{1}{t^\kappa}-\underline{\mu}\le0,~\forall t\ge B_4$ and using (\ref{online_op:dyz}) and (\ref{online_op:e3}) yield
\begin{align}\label{online_op:e4}
&\sum_{t=1}^TE_{3}(t+1)=\sum_{t=1}^{B_4-1}E_{2}(t+1)+\sum_{t=B_4}^TE_{3}(t+1)\nonumber\\
&\le\frac{nd(X)K}{\alpha_{B_4}}\nonumber\\
&~~~+\sum_{t=B_4}^T\sum_{i=1}^n\Big[\frac{1}{\alpha_{t}}\calD_{\psi_i}(y_{i,t},x_{i,t})
-\frac{1}{\alpha_{t+1}}\calD_{\psi_i}(y_{i,t+1},x_{i,t+1})\Big]\nonumber\\
&~~~+\sum_{t=B_4}^T\sum_{i=1}^n\Big(\frac{1}{\alpha_{t+1}}
-\frac{1}{\alpha_{t}}-\underline{\mu}\Big)\calD_{\psi_i}(y_{i,t},x_{i,t})\nonumber\\
&\le\frac{2nd(X)K}{\alpha_{B_4}}.
\end{align}
Replacing (\ref{online_op:dyz}) with (\ref{online_op:e4}) and along the same line as the proof of (\ref{online_op:corollaryregequ1}) in Theorem~\ref{online_op:corollaryreg} give that $\Reg(\bsx_T,\check{\bsx}_T^*)\le C_4T^{\kappa}$ when $\alpha_t=\frac{1}{t^{1-\kappa}}$.

Next, we show that (\ref{online_op:regstongconvex}) holds.  When $\kappa\in(0,0.5)$, we have $\alpha_t=1/t^{(1-\kappa)}$. Thus, from the above result, we have $\Reg(\bsx_T,\check{\bsx}_T^*)\le C_4T^{\kappa}$. When $\kappa\in[0.5,1)$, we have $\alpha_t=1/t^{\kappa}$. Thus, (\ref{online_op:staticregequ1}) gives $\Reg(\bsx_T,\check{\bsx}_T^*)\le C_1T^{\kappa}$. In conclusion, (\ref{online_op:regstongconvex}) holds.

(ii)  Substituting $c=1-\kappa$ when $\kappa\in(0,0.5)$ and $c=\kappa$ when $\kappa\in[0.5,1)$ in (\ref{online_op:corollaryconsequ}) gives (\ref{online_op:regcstongconvex}).

%\subsection{Proof of Theorem~\ref{online_op:theoremstongconvexslater}}\label{online_op:theoremstongconvexslaterproof}

\end{document}

%% file: styledsymbols.tex
\def\bsx{{\boldsymbol{x}}}
\def\bsy{{\boldsymbol{y}}}

\def\calD{{\mathcal{D}}}

\def\calP{{\mathcal{P}}}

\def\calS{{\mathcal{S}}}

\def\calX{{\mathcal{X}}}

\def\Real{{\mathbb{R}}}